\documentclass{mcom-l}

\newtheorem{theorem}{Theorem}[section]
\newtheorem{lemma}[theorem]{Lemma}

\theoremstyle{definition}

\theoremstyle{remark}

\numberwithin{equation}{section}
\usepackage{graphicx}
\usepackage{subcaption}
\usepackage{amssymb,amsmath,amsbsy,mathrsfs,eucal,amsfonts}

\usepackage{multirow}
\usepackage{color}
\definecolor{blck}{rgb}{0,0,0}

\definecolor{darkred}{rgb}{.6,.1,0}

\definecolor{blue}{rgb}{0,0,1}

\definecolor{red}{rgb}{1,0,0}

\usepackage{tikz}

\newcommand{\jmp}[1]{[\![#1]\!]}                     

\newcommand{\Eh}{\mathscr{E}_h}

\def\calT{{\mathcal T}}

\def\eps{{\epsilon}}

\DeclareMathOperator{\sech}{sech}
\DeclareMathOperator{\csch}{csch}

\begin{document}

\title{Optimally convergent HDG method for third-order Korteweg-de Vries type equations}




\author{Bo Dong}
\address{Department of
Mathematics, University of Massachusetts Dartmouth, 285 Old Westport Road, North Dartmouth, MA 02747, USA}
\curraddr{} \email{bdong@umassd.edu}


\date{}

\dedicatory{}

\begin{abstract}
We develop and analyze a new hybridizable discontinuous
Galerkin (HDG) method for solving third-order Korteweg-de Vries type equations. The approximate solutions are defined by a discrete version of a characterization of the exact solution in terms of the solutions to local problems on each element which are patched together through transmission conditions on element interfaces.
We prove that the semi-discrete scheme is stable with proper choices of stabilization function in the numerical traces. For the linearized equation, we carry out error analysis and show that the approximations to the exact solution and its derivatives have optimal convergence rates. In numerical experiments, we use an implicit scheme for time discretization and the Newton-Raphson method for solving systems of nonlinear equations, and observe optimal convergence rates for both the linear and the nonlinear third-order equations.
\end{abstract}

\maketitle

\pagestyle{myheadings} \thispagestyle{plain} \markboth{
B. DONG}
{HDG method for KdV equations}

\section{Introduction}

In this paper, we develop and analyze a new hybridizable
discontinuous Galerkin (HDG) method for the following initial-boundary value problem of the Korteweg-de Vries (KdV) type equation on a finite domain 
\begin{equation}
\label{eq:prob}
\begin{split}
u_t+u_{xxx} + F(u)_x\,&=\,f\;\; \quad\quad   {\textrm{ for } }\; x\in\Omega:=(a,\;b), t\in (0, T],\\
u\,&=\,u_0\quad\quad  {\textrm{ in } }\Omega \textrm{ for }t=0,\\
u \,&=\, u_D \quad\quad  {\textrm {on} }\; \partial \Omega:= \{a,b\},\\
u_x \,&=\, q_N \quad\quad {\textrm {on} }\;\partial \Omega_N:=\{b\}.
\end{split}
\end{equation}
Here $f \in L^2(\Omega)$ and $F(u)=\beta u^m$, where $\beta$ is a constant and $m\ge 0$ an integer. 
The well-posedness of the problem \eqref{eq:prob} and properties of the solution have been theoretically and numerically studied; see \cite{BonaSunZhang03,Holmer06,BonaChenSunZhang07,BonaSunZhang09,GoubetShen07,SkogestadKalisch09} and references therein.

KdV type equations play an important
role in applications, such as fluid mechanics \cite{Phillips74,Buckingham97,PandaDawsonZhangKennedyWesterinkDonahue14}, nonlinear optics \cite{BiswasRahmanDas11,Horsley16}, acoustics \cite{Schamel73,Tran79}, plasma physics \cite{Braginskii65,ZabuskyKruskal65,TassiMorrisonWaelbroeckGrasso08,ShuklaEliasson11}, and Bose-Einstein condensates \cite{Tagare73,KamchatnovShchesnovich04} among other fields. They also have an enormous impact on the development of nonlinear mathematical
science and theoretical physics. Many modern areas were opened up as a consequence of the basic research on KdV equations. Due to their importance in applications and
theoretical studies, there has been a lot of interest in developing accurate and efficient numerical methods for KdV equations. In particular, an ongoing effort on developing discontinuous Galerkin (DG) methods
for KdV type equations has been made in the last decade. The first DG method, the local discontinuous Galerkin (LDG) method, for the KdV equation was introduced in 2002 by Yan and Shu in \cite{YanShu02} and further studied for the linear case in \cite{LiuYan06,XuShu07,XuShu12,HuffordXing14}.
In \cite{ChengShu08}, a DG method for the $KdV$ equation was devised by using
repeated integration by parts.
Recently, several conservative DG methods \cite{BonaChenKarakashianXing13,ChenCockburnDongDG16,KarakashianXing16} were developed for KdV type equations to
preserve quantities such as the mass and the $L^2$-norm of the solutions. 
When solving KdV equations, one can use these DG methods for spatial discretization together with explicit schemes for time-marching if the coefficient before the third-order derivative is very
small. However, when such coefficient is of order one, for example, implicit
time-marching methods might be the methods of choice.

Traditional DG methods, despite their prominent features such as $hp$-adaptivity and local conservativity, were criticized for having larger number of
degrees of freedom than continuous finite element methods  when solving steady-state problems or problems that require implicit-in-time solvers.
Here, we develop an HDG method which is very suitable for solving KdV equations when implicit time-marching is used.
HDG methods \cite{CockburnGopalakrishnanLazarov09, CockburnDongGuzmanSFH08,CockburnGuzmanWang09,CockburnGopalakrishnanSayas10} were first introduced for diffusion problems and they provide optimal approximations to both the potential and the flux.
Due to the feature that the global coupled degrees of freedom only live on element interfaces, they are significantly advantageous for solving steady-state problems or time-dependent problems that require implicit time-marching.
In \cite{ChenCockburnDongHDG16}, we introduced the first family of HDG methods for stationary third-order linear equations, 
which allow the approximations to the exact solution $u$ and its derivatives $u_x$ and $u_{xx}$ to have different polynomial degrees.
We proved superconvergence properties of these methods on projection of errors and numerical traces, and numerical results indicate that 
the HDG method using the same polynomial degree $k$ for all three variables is quite robust with respect to the choice of the stabilization function
and provides a converging postprocessed solution with order $2k+1$ with the least amount of degrees of freedom.
This suggests that the HDG method using the same polynomial degrees for all variables is the method of choice for solving one-dimensional third-order problems.
Therefore, in this paper we extend this HDG method to time-dependent third-order  KdV type equations.

To construct the HDG method for KdV equations, we follow the
approach used in   \cite{ChenCockburnDongHDG16} for stationary third-order equations. That is, given any mesh of the domain,
we show that the exact solution can be obtained by solving the equation on each element with provided boundary data that are determined by transmission conditions.
Then we define  HDG methods by a discrete version of this characterization, which ensures that the only globally-coupled degrees of freedom are those
associated to the numerical traces on element interfaces. In \cite{ChenCockburnDongHDG16}, it was shown that HDG methods derived by
providing boundary data to local problems in different ways are indeed equivalent to each other when the stabilization function is finite and nonzero.
So here we just need to consider the one that takes
the numerical trace of $u$ at both ends
of the interval and the numerical trace of $u_{xx}$ at the right end as boundary data for the local problems.
Our method is different from the HDG method in \cite{SamiiPandaMichoskiDawson16}, which was designed from implementation 
 point of view.
That HDG method involves two sets of numerical traces for $u_x$,  and there is no error analysis for the method.

Our way of devising HDG methods from the characterization of the exact solution allows us to carry out stability and error analysis. We first apply an energy
argument to find conditions on the stabilization function in the numerical traces, under which the HDG method has a unique solution for KdV type equations. Then by deriving four energy identities and combining them together, we prove that the method has optimal approximations
to $u$ as well as its derivatives $u_x$ and $u_{xx}$ for linear equations; this technique is similar to that in \cite{XuShu12}.  In implementation,  implicit  time-marching
schemes such as BDF or DIRK methods can be used, and at each time step a stationary third-order equation is solved by the HDG method together with the Newton-Raphson method (see Appendix A).
Due to the one-dimensional setting of the KdV equations, the global
matrix of the HDG method that needs to be numerically inverted at each time step is independent of the polynomial degree of the approximations,
its size is only $2N+1$, where $N$ is the number of intervals of the mesh, and its condition number is
of the order of $h^{-2}$, where $h$ denotes the size of the intervals of the mesh.


The paper is organized  as follows. In Section \ref{sec:mainresults}, we define the HDG method for third-order KdV type equations and
state and discuss our main results. The details of all the proofs are given in Section \ref{sec:proofs}.
We show numerical results in Section \ref{sec:numericaltest} and some concluding remarks in Section \ref{sec:conclude}. The details on implementation of the method are in Appendix A.

\section{Main Results}
\label{sec:mainresults} In this section, we state and discuss our main results. We begin by describing the characterizations of the exact solution that the HDG method
is a discrete version of. We then introduce our HDG method for KdV type equations, and state our stability result and optimal a priori error estimate.

\subsection{Characterizations of the exact solution}
To display the characterizations of the exact solution we are going to work with, let us first rewrite our third-order model equation as the following first-order system:
\begin{subequations}
\label{eq:mixedprob}
\begin{alignat}{1}
q - u_x \,&=\, 0, \quad p - q_x\,=\, 0,\quad u_t+ p_x + F(u)_x\,=\,f \quad{\textrm {for} }\;x\in\Omega, t\in (0, T],\\
\intertext{with the initial and boundary conditions}
u\,&=\,u_0\quad\quad  {\textrm{ in } }\Omega \textrm{ for }t=0,\\
u \,&=\, u_D \quad\quad  {\textrm {on} }\; \partial \Omega, \\
q \,&=\, q_N \quad\quad {\textrm {on} }\;\partial\Omega_N.
\end{alignat}
\end{subequations}
We partition  the domain $\Omega$ as
\[
\calT_h=\{I_i:=(x_{i-1}, x_i): a=x_0<x_1<\cdots<x_{N-1}<x_N=b\},
\]
and introduce the set  of the boundaries of its elements, $\partial\calT_h:=\{ \partial I_i: i=1,\dots,N\}$. We also set $\Eh:=\{x_i\}_{i=0}^N$,
$h_i = x_i - x_{i-1}$ and  $h:=\max_{i=1}^N h_i$.

We know that, when $f$ is smooth enough, if we provide the values $\{\widehat{u}_i\}_{i=0}^N$ and $\{\widehat{p}_i\}_{i=1}^N$ and, for each $i=1,\dots,N$, solve the local problem
\begin{alignat*}{1}
&
Q - U_x=0,
 \quad
P - Q_x=0,
 \quad
U_t+P_x + F(U)_x=f
 \quad\text{ in } I_i,
\\
&
U =u_0\;\; \text{ for } t=0,
 \quad
U(x^+_{i-1})=\widehat{u}_{i-1},
 \quad
U(x^-_{i})=\widehat{u}_{i},
 \quad
P(x^-_{i})=\widehat{p}_{i},
\end{alignat*}
then $(P,Q,U)$ coincides with the solution $(p,q,u)$ of \eqref{eq:mixedprob} if and only if the transmission conditions
\begin{equation*}
Q(x_i^-)=Q(x_i^+),\quad
P(x_i^-)=P(x_i^+),\qquad i=1,\dots,N-1
\end{equation*}
and the boundary conditions
\begin{equation*}
U=u_D \;\;{\textrm {on} }\; \partial \Omega, \quad \quad Q =q_N \;\; {\textrm {on} }\; \partial \Omega_N
\end{equation*}
are satisfied. There are other possible characterizations of the exact solution corresponding to different choices of boundary data for the local problem;
see \cite{ChenCockburnDongHDG16}. Note that
for these characterizations, the boundary data of the local problems are the unknowns of a global problem obtained from the transmission conditions and boundary conditions,
and the system of equations for the global unknowns is square.

\subsection{HDG method}
To define our HDG method, we first introduce the finite element spaces to be used. We let the
approximations $(u_h, q_h, p_h, \widehat{u}_h, \widehat{q}_h, \widehat{p}_h)$ to
$(u|_\Omega,q|_\Omega,p|_\Omega,u|_{\Eh},q|_{\Eh},p|_{\Eh})$ be in the space
$W_h^k\times W_h^k \times W_h^k \times L^2(\Eh)\times L^2(\partial\calT_h) \times L^2(\partial\calT_h)$ where
\begin{alignat*}{3}
{W}_h^k & = \{w\in L^2(\mathcal{T}_h):&&\quad
   w|{_{I_i}} \in {P}_{k}(I_i)&&\quad\forall\; i=1,\cdots,N\}.
\end{alignat*}
Here {$P_k(I_i)$ is the space of polynomials of degree at most $k$ on the domain $I_i$}. For any function $\zeta$ lying in $L^2(\partial\calT_h)$, we denote
its values on $\partial I_i:=\{x_{i-1}^+, x_i^-\}$ by $\zeta(x_{i-1}^+)$ (or simply $\zeta^+_{i-1}$) and $\zeta(x_i^-)$ (or simply $\zeta^-_i$).
Note that $\zeta(x_{i}^+)$ is not necessarily equal to
$\zeta(x_i^-)$. In contrast, for any $\eta$ in the space $L^2(\Eh)$, its value at $x_i$, $\eta(x_i)$ (or simply $\eta_i$) is uniquely defined; in this case,
 $\eta(x_i^-)$ or $\eta(x_{i}^+)$ mean nothing but $\eta(x_i)$.

To obtain the HDG formulation, we use a discrete version of the characterization of the exact solution. Assuming that the values  $\{\widehat{u}_{hi}\}_{i=0}^N$ and $\{\widehat{p}^{\;-}_{hi}\}_{i=1}^N$ are given, for each $i=1,\dots,N$, we solve
a local problem on the element $I_i$ by using a Galerkin method. To describe it, let us introduce the following notation.  By $(\varphi, v)_{I_i}$, we denote the integral of $\varphi$ times $v$ on the interval $I_i$, and by $\langle \varphi,v n\rangle_{\partial I_i}$ we simply mean the expression $\varphi(x_i^-)v(x_i^-)n(x_i^-)
+\varphi(x_{i-1}^+)v(x_{i-1}^+)n(x_{i-1}^+)$. Here $n$ denotes the outward unit normal to $I_i$: $n(x_{i-1}^+):=-1$ and $n(x_i^-):=1$.

On the element $I_i=(x_{i-1}, x_i)$, we give $f$ and the boundary data $\widehat{u}_{h\, i-1}, \widehat{u}_{h\,i}$
and $\widehat{p}^{\; -}_{h\, i}$ and take the HDG approximate solutions $(p_h,q_h,u_h)\in P_{k}(I_i)\times P_{k}(I_i)\times P_{k}(I_i)$ to be the solution of the equations
\begin{alignat*}{2}
({q}_h,{v})_{I_i} + (u_h,v_x)_{I_i} - \langle \widehat{u}_h,{v}n \rangle_{\partial I_i} & = 0, \\
({p}_h,{z})_{{I_i}} + (q_h,z_x)_{I_i} - \langle \widehat{q}_h,{z}n \rangle_{\partial I_i} & = 0,\\
({{u}_h}_t,{w})_{I_i}-(p_h+F(u_h),w_x)_{I_i} + \langle \widehat{p}_h+\widehat{F}_h,{w}n \rangle_{\partial I_i} & = (f,{w})_{I_i},
\end{alignat*}
for all $(v, z, w)\,\in\,P_{k}(I_i)\times P_{k}(I_i) \times P_{k}(I_i)$, where the remaining undefined numerical traces are given by
\begin{alignat*}{2}
&\begin{cases}
\widehat{p}_h=p_h  + \tau_{pu}\,({\widehat{u}_{h\,i-1}} - u_h)\,n &\mbox{ at } x_{i-1}^+,\\
\widehat{q}_h=q_h  + \tau_{qu}\,( {\widehat{u}_{h\, i-1}} - u_h)\,n &\mbox{ at } x_{i-1}^+,\\
\widehat{q}_h= q_h  + \tau_{qu} \,( {\widehat{u}_{h\, i}} - u_h)\,n +\tau_{qp}\,( {\widehat{p}^{\;-}_{h\,i}} -p_h)\,n&\mbox{ at } x_{i}^-,\\
\widehat{F}_h=F(\widehat{u}_h) -\tau_F(\widehat{u}_h, u_h) (\widehat{u}_h -u_h)\, n & \mbox{ at } x_{i-1}^+ \mbox{ and } x_i^-.\\
\end{cases}
\end{alignat*}
The functions $\tau_{qu}, \tau_{pu}, \tau_{qp}$, and $\tau_F(\widehat{u}_h, u_h)$
are defined on $\partial\calT_h$ and are called the components of the {\em stabilization function}; they have to be properly chosen to ensure that the above problem has a unique solution. In particular, due to the nonlinearity of $F$, the function $\tau_F(\cdot,\cdot): \partial\mathcal{T}_h\rightarrow \mathbb{R}$ can be nonlinear in terms of $\widehat{u}_h$ and $u_h$. In the case of $F=0$, we simply take $\tau_F=0$.

It remains to impose the transmission conditions
\[
\jmp{\widehat{q}_h } (x_i) = 0
\quad\mbox{ and }\quad
\jmp{\widehat{p}_h+\widehat{F}_h} (x_i) = 0
\qquad \mbox{ for all } i =1,\dots,N-1,
\]
and the boundary conditions
\[
\widehat{u}_h  = u_D  \;\;{\textrm {on} }\; \partial \Omega
\quad\mbox{ and }\quad
\widehat{q}_h  = q_N \;\;{\textrm {on} }\; \partial \Omega_N.
\]
Here, $\jmp{\zeta}(x_i):=\zeta(x_i^-)-\zeta(x_i^+)$. This completes the definition of the HDG methods using the characterization of the exact solution. Note that this way of defining the HDG methods immediately provides a way to implement them.

{On} the other hand, the above presentation of the HDG methods is not very well suited for their analysis. 
Thus, we now rewrite it in a more compact form using the notation
$$(\varphi, v):=\sum_{i=1}^N (\phi, v)_{I_i},  \quad \langle \varphi, v n\rangle:=\sum_{i=1}^N \langle \varphi, v n\rangle_{\partial I_i}.$$
Let
$$M_h(g):=\{ \zeta\in L^2(\Eh): \;\; \zeta|_{\partial \Omega}=g\},\qquad \tilde{M}_h:= L^2(\Eh\setminus\{a\}).
$$
The approximation provided by the HDG method,  $(u_h, q_h, p_h, \widehat{u}_h, \widehat{p}_h^{\,-})$, is the element  of
$W_h^{k}\times W_h^{k} \times W_h^{k} \times M_h(u_D)\times  \tilde{M}_h$ which solves the equations
\begin{subequations}
\label{eq:Nmethod}
\begin{alignat}{2}
\label{eq:Nmethod1}
({q}_h,{v})  + (u_h,v_x)  - \langle \widehat{u}_h,{v}n \rangle & = 0, \\
\label{eq:Nmethod2}
({p}_h,{z})  + (q_h,z_x)  - \langle \widehat{q}_h,{z}n \rangle  & = 0,\\
\label{eq:Nmethod3}
({u_h}_t,{w}) -(p_h+F(u_h),w_x) + \langle \widehat{p}_h+\widehat{F}_h,{w}n \rangle  & = (f,{w}),\\
\intertext{and}
\label{eq:Nmethod5}
\langle \widehat{q}_h, \mu n\rangle = \langle q_N, \mu n\rangle_{\partial\Omega_N},
\quad
\langle \widehat{p}_h+\widehat{F}_h, \chi n\rangle =& 0
\end{alignat}
for all $(v, z, w,\mu,\chi)\,\in\,W_h^{k}\times W_h^{k} \times W_h^{k}\times \tilde{M}_h\times M_h(0)$, where, on $\partial\calT_h$, we have
\begin{equation}
\label{eq:Nmethod4}
\begin{cases}
\widehat{p}_h^{\,+}=p_h^+ + \tau_{pu}^+\,(\widehat{u}_h - u_h^+)\,n^+,&\\
\widehat{q}_h^{\,+}=q_h^+  + \tau_{qu}^+\,(\widehat{u}_{h} - u_h^+)\,n^+,\\
\widehat{q}_h^{\,-}=q_h^- + \tau_{qu}^-\,(\widehat{u}_h - u_h^-)\,n^- +\tau_{qp}^-\,({\widehat{p}_h^{\,-}} -p_h^-)\,n^-,\\
\widehat{F}_h= F(\widehat{u}_h) -\tau_F(\widehat{u}_h,u_h)\,(\widehat{u}_h -u_h)\,n.
\end{cases}
\end{equation}
\end{subequations}


It is not difficult to define HDG methods that are associated to other characterizations of the exact solution, but
these methods 
are actually the same, provided
that the corresponding stabilization function allows for the transition from one characterization to the other; see  \cite{CockburnGopalakrishnan09,ChenCockburnDongHDG16}. In fact, the choice of characterization to use is more relevant for the actual implementation of the HDG method
rather than for its actual definition.
The implementation of the HDG method \eqref{eq:Nmethod} is discussed in the Appendix.

When above scheme is discretized in time, we can choose the initial approximation ($u_h^0, q_h^0, p_h^0, \widehat{u}_h^0, \widehat{p}_h^0$) to be the HDG approximate solutions of
the stationary equation $v + v_{xxx} + F(v)_x = g,$ where $g=u_0+(u_0)_{xxx}+F(u_0)_x$ and $u_0$ is the initial data of the time-dependent problem \eqref{eq:prob}; see \cite{ChenCockburnDongHDG16} for HDG methods on stationary third-order equations.
The initial approximation $(u_h^0, q_h^0, p_h^0, \widehat{u}_h^0, \widehat{p}_h^{0})$, is the element  of
$W_h^{k}\times W_h^{k} \times W_h^{k} \times M_h(u_D)\times  \tilde{M}_h$ which solves the equations
\begin{alignat*}{2}
({q}_h^0,{v})  + (u_h^0,v_x)  - \langle \widehat{u}_h^0,{v}n \rangle & = 0, \\
({p}_h^0,{z})  + (q_h^0,z_x)  - \langle \widehat{q}_h^0,{z}n \rangle  & = 0,\\
({u_h^0},{w}) -(p_h^0+F(u_h^0),w_x) + \langle \widehat{p}_h^0+\widehat{F}_h^0,{w}n \rangle  & = (g,{w}),\\
\langle \widehat{q}_h^0, \mu n\rangle = \langle q_N, \mu n\rangle_{\partial\Omega_N},
\quad
\langle \widehat{p}_h^0+\widehat{F}_h^0, \chi n\rangle &= 0
\end{alignat*}
for all $(v, z, w,\mu,\chi)\,\in\,W_h^{k}\times W_h^{k} \times W_h^{k}\times \tilde{M}_h\times M_h(0)$, where $\widehat{q}_h^0, \widehat{p}_h^0$, and $\widehat{F}_h^0$ are defined in the same ways as $\widehat{q}_h, \widehat{p}_h$, and $\widehat{F}_h$ in \eqref{eq:Nmethod4}.
 Note that the equations above are almost the same as those in  \eqref{eq:Nmethod} except the third one. 
This way of choosing initial data for time-dependent problems by
solving corresponding stationary problems has been used in \cite{CockburnFuJiSayas, ChenCockburnDongDG16}.

Next, we present our stability result and a priori error estimate of the HDG method  under some conditions on
 the stabilization function. 

\subsection{Stability}

To discuss the $L^2$-stability of the HDG method, we let
$$\tilde{\tau}(u_h, \widehat{u}_h):= \frac{1}{(u_h-\widehat{u}_h)^2}{\int_{\widehat{u}_h}^{u_h}(F(s)-F(\widehat{u}_h)) n\, ds}.$$
We have the following stability result.

\begin{theorem}\label{thm:energy2}
Assume that $u_D=q_N=0$. If the stabilization function satisfies
\begin{equation}\label{eq:Ntau_cond}
\begin{split}
 &(\tau_F^+-\tilde{\tau}^+) -\tau_{pu}^+ -\frac{1}{2}(\tau_{qu}^+)^2 \ge 0,\textrm{ and } \\
  &(\tau_F^--\tilde{\tau}^-) +\frac{1}{2}(\tau_{qu}^-)^2\ge 0, 
  \quad (\tau_F^--\tilde{\tau}^-)(\tau_{qp}^-)^2+\tau_{qu}^-\tau_{qp}^- -\frac{1}{2}\ge 0,
\end{split}
\end{equation}
then for the HDG method \eqref{eq:Nmethod}, we have
$$ \frac{d}{dt}\|u_h\|^2 \le 2(f,u_h)
.$$
\end{theorem}
Note that if the nonlinear term $F=0$, then we have $\tau_F =\tilde{\tau}=0$ and the condition \eqref{eq:Ntau_cond} in the Theorem above can be simplified as
 \begin{alignat}{1}\label{eq:tau_cond}
 -\tau_{pu}^+ -\frac{1}{2}{(\tau_{qu}^+)}^2\ge 0\quad\textrm{ and } \qquad\tau_{qu}^-\tau_{qp}^- -\frac{1}{2}\ge 0.
\end{alignat}
If $F(u)\neq 0$, we just need to have $\tau_F\ge \tilde{\tau}$ and take $\tau_{qu}^\pm, \tau_{pu}^+$ and $\tau_{qp}^-$ to satisfy \eqref{eq:tau_cond}. Since
$$\tilde{\tau}=\frac{1}{(u_h-\widehat{u}_h)^2}{\int_{\widehat{u}_h}^{u_h}F'(\xi)(s-\widehat{u}_h) n\, ds}\le \frac{1}{2} \sup_{s\in J(u_h, \widehat{u}_h)}|F'(s)|,$$
where $J(u_h, \widehat{u}_h)=[\min\{u_h,\widehat{u}_h\}, \max\{u_h,\widehat{u}_h\}]$, the stabilization function $\tau_F$ satisfies the condition $\tau_F\ge \tilde{\tau}$ if
$$\tau_F\ge \frac{1}{2}\sup_{s\in J(u_h, \widehat{u}_h)}|F'(s)|.$$
For other choices of $\tau_F$ which satisfies the condition  $\tau_F\ge \tilde{\tau}$, see \cite{NguyenPeraireCockburn09}.

\subsection{A priori error estimate for linear equations}
Now we consider the convergence properties of our HDG method for linear equations in which $F=0$. We proceed as follows. We first define
an auxiliary projection and state its optimal approximation property. Then,
we provide an estimate for the $L^2$-norm of the projections of the errors in the
primary and auxiliary variables.

Let us introduce a key auxiliary projection that is tailored to the numerical traces. The projection of the function $(u,q,p)\in
H^1(\mathcal{T}_h)\times H^1(\mathcal{T}_h)\times H^1(\mathcal{T}_h)$,
$\varPi (u, q, p) := (\varPi u, \varPi {q},\varPi p)$, is defined as follows.
{On an element ${I_i}=(x_{i-1}, x_i)$, the projection is the element of ${{ P }}_{k}({I_i})\times{ { P }}_{k}({I_i})\times P_{k}({I_i})$ which solves the following equations:}
\begin{subequations}
\label{eq:proj}
 \begin{alignat}{1}
   \label{eq:projq}
  (\delta_u,v)_{I_i} & = 0  \quad\forall \,\,v \in { P }_{k-1}({I_i}),\\
    \label{eq:projp}
  (\delta_q,z)_{I_i} & = 0  \quad\forall \,\,z \in { P }_{k-1}({I_i}),\\
  \label{eq:proju}
  (\delta_p,w)_{I_i} & = 0  \quad\forall \,\,w \in { P }_{k-1}({I_i}),\\
  \label{eq:projpu}
\delta_p - \tau_{pu}^+\,\delta_u \,n  & =0 \quad\mbox{ on }  x_{i-1}^+,\\
  \label{eq:projqu+}
\delta_q - \tau_{qu}^+\,\delta_u\,n   & =0 \quad\mbox{ on } x_{i-1}^+,\\
  \label{eq:projqu-}
\delta_q - \tau_{qu}^-\,\delta_u\,n -\tau_{qp}^-\,\delta_p\, n & =0 \quad\mbox{ on } x_i^-,
 \end{alignat}
\end{subequations}
where we use the notation $\delta_\omega := \omega - \varPi \omega$ for $\omega = u, q$, and $p$. Note that the last three equations have exactly the same
structure as the numerical traces of the HDG method in \eqref{eq:Nmethod4}.

The following result for the optimal approximation properties of the
projection $\varPi$ was shown in \cite{ChenCockburnDongHDG16}. To state it, we use the following notation. The $H^s(D)$-norm is denoted by $\|\cdot\|_{s, D}$.
We drop the first subindex if $s=0$,  and the second one if $D=\Omega$ or $D={\mathcal T}_h$.
\begin{lemma}
\label{lemma:projapprox}
Suppose that
\begin{equation}
\label{eq:cond_tau}
\tau_{qu}^+ +\tau_{qu}^- -\tau_{pu}^+\tau_{qp}^- \neq 0.
\end{equation}
Then the projection $\varPi$ in \eqref{eq:proj} is well defined on any interval $I_i$. In addition, if 
$\tau_{qu}^+, \tau_{qu}^-, \tau_{pu}^+$ and $\tau_{qp}^-$ are constants, we
have that, for $\omega=u,q$ and $p$, there is a constant $C$ such that
\begin{alignat*}{1}
\|\omega-\varPi \omega\|_{I_i} &\le C\, h^{s+1}
\quad\mbox{ for }s\in [1, {k}],
\end{alignat*}
provided $\omega\in H^{s+1}(I_i)$.
\end{lemma}

Next, we provide estimates for the $L^2$-norm of the projection of the errors
\[
\epsilon_u :=\; \varPi u - u_h,\quad \epsilon_{q} :=\; \varPi {q} - {q}_h, \quad\epsilon_{p} :=\; \varPi {p} - {p}_h,
\]
and deduce from them  {the} estimates for the $L^2$-norm of the errors
\[
e_u :=\; u - u_h,\quad e_{q} :=\;{q} - {q}_h, \quad e_{p} :=\; {p} - {p}_h.
\]

\begin{theorem}
\label{thm:new}
Suppose that $F(u)=0$ in the problem \eqref{eq:mixedprob} and the exact solution $(u,q,p)\in W^{2,\infty}((0, T]; H^{k+1}(\mathcal{T}_h))\times
                          W^{1,\infty}((0, T]; H^{k+1}(\mathcal{T}_h))\times
                          W^{1,\infty}((0, T]; H^{k+1}(\mathcal{T}_h))$.
If the stabilization function 
 of the HDG method \eqref{eq:Nmethod}
satisfies the condition 
\begin{equation}\label{eq:tau_cond_new}
\begin{split}
&\tau_{qu}^->0, \,\,\tau_{qu}^-\tau_{qp}^-=1, \textrm{ and }\\
 &  \tau_{qu}^+\in [0, 1], \,\,\tau_{pu}^+\in [-1-\sqrt{1-{\tau_{qu}^+}^2},-\frac{1}{2}-\frac{1}{2}{\tau_{qu}^+}^2],
\end{split}
\end{equation}
 then for $k>0$ and $h$ small enough, we have
$$\|\epsilon_u(t)\|+ \|\epsilon_q(t)\| + \|\epsilon_p(t)\|+\|{\epsilon_u}_t(t)\|\le Ch^{k+1}\quad \textrm{ for } 0\le t\le T,
$$
where $C$ is independent of $h$.
\end{theorem}

It is easy to see that if the stabilization function satisfies the condition \eqref{eq:tau_cond_new}, then it also satisfies the conditions \eqref{eq:tau_cond} and \eqref{eq:cond_tau}.
Using Lemma \ref{lemma:projapprox}, Theorem \ref{thm:new} and the triangle inequality, we immediately get the following $L^2$ error estimate for the actual errors.
\begin{theorem}\label{thm:actual_error}
Suppose that the hypotheses of Theorem \ref{thm:new} are satisfied. Then we have
\[
\|e_u(t)\|+\|e_q(t)\|+\|e_p(t)\|+\|e_{u_t}(t)\|\le Ch^{k+1} \quad \textrm{ for } 0\le t\le T,
\]
where $C$ is independent of $h$.
\end{theorem}

\section{Proofs}
\label{sec:proofs}
In this section, we provide detailed proofs of our main results. We first prove Theorem \ref{thm:energy2} on the  $L^2$-stability of the HDG method for general KdV type equations. Then we combine several energy identities to prove the error estimate in Theorem \ref{thm:new} for linear third-order  equations.

%

\subsection{$L^2$-stability}
 Now let us prove Theorem \ref{thm:energy2} on the stability of the HDG method for the KdV equation. We treat the nonlinear term  in
 a way similar to that in \cite{NguyenPeraireCockburn09}.
 \begin{proof}
Taking $\omega=u_h, v=-p_h$ and $z=q_h$ in \eqref{eq:Nmethod1}--\eqref{eq:Nmethod3} and adding the three equations together, we get
\begin{alignat*}{1}
(f, u_h)=&({u_h}_t,u_h)-(p_h+F(u_h), {u_h}_x) + \langle \widehat{p}_h+\widehat{F}_h, u_h n\rangle\\
&-(q_h, p_h)- (u_h, {p_h}_x) + \langle \widehat{u}_h, p_h n\rangle\\
&+(p_h,q_h)+(q_h, {q_h}_x)-\langle \widehat{q}_h, q_h n\rangle.
\end{alignat*}
Using integration by parts and \eqref{eq:Nmethod5}, we have
\begin{equation}\label{eq:energy1_1}
\begin{split}
(f,u_h)=&\frac{1}{2}\frac{d}{dt}\|u_h\|^2 -(F(u_h),{u_h}_x)-\langle \widehat{p}_h+\widehat{F}_h-p_h, (\widehat{u}_h-u_h)n \rangle \\
&+\frac{1}{2}\langle (\widehat{q}_h-q_h)^2, n\rangle + \frac{1}{2}\widehat{q}_h^{\, 2}(0).
\end{split}
\end{equation}
Let $G(s)$ be such that $d G(s)/ds=F(s)$. It is easy to see that
\begin{alignat*}{1}
-(F(u_h),{u_h}_x)&= -(\frac{d}{dx} G(u_h), 1) = -\langle G(u_h), n\rangle
=-\langle \int_{\widehat{u}_h}^{u_h} F(s) ds, n\rangle. 
\end{alignat*}
Using it for the second term on the right hand side of \eqref{eq:energy1_1}, we get that
\begin{alignat*}{1}
(f,u_h)=&\frac{1}{2}\frac{d}{dt}\|u_h\|^2+ \Phi 
+ \frac{1}{2}\widehat{q}_h(0)^2,
\intertext{where}
\Phi=& -\langle \int_{\widehat{u}_h}^{u_h} (F(s)-F(\widehat{u}_h)) ds, n\rangle   -\langle \widehat{F}_h-F(\widehat{u}_h), (\widehat{u}_h-u_h)n \rangle \\
&-\langle \widehat{p}_h-p_h, (\widehat{u}_h-u_h)n \rangle+\frac{1}{2}\langle (\widehat{q}_h-q_h)^2, n\rangle.
\end{alignat*}
Next, we just need to show that $\Phi\ge 0$. Let $$\tilde{\tau}:=\frac{1}{(\widehat{u}_h-u_h)^2}\int_{\widehat{u}_h}^{u_h} (F(s)-F(\widehat{u}_h))n ds.$$ Using the definition of $\widehat{F}_h$ in \eqref{eq:Nmethod4}, we have
$$\Phi=\langle \tau_F-\tilde{\tau},(\widehat{u}_h-u_h)^2 \rangle -\langle \widehat{p}_h-p_h, (\widehat{u}_h-u_h)n\rangle +\frac{1}{2}\langle (\widehat{q}_h-q_h)^2, n\rangle.
$$
By the definition of $\widehat{p}_h$ and $\widehat{q}_h$ in \eqref{eq:Nmethod4}, we get
\begin{alignat*}{1}
\Phi^+ :=\Phi|_{\partial\mathcal{T}_h^+}
       =&\langle \tau_F^+-\tilde{\tau}^+-\tau_{pu}^+ -\frac{1}{2}(\tau_{qu}^+)^2, (\widehat{u}_h-u_h)^2 \rangle_{\partial\mathcal{T}_h^+},\\
\Phi^- :=\Phi|_{\partial\mathcal{T}_h^-}= &\langle \tau_F^- -\tilde{\tau}^-+\frac{1}{2}(\tau_{qu}^-)^2, (\widehat{u}_h-u_h)^2 \rangle_{\partial\mathcal{T}_h^-}+\langle \frac{1}{2}(\tau_{qp}^-)^2, (\widehat{p}_h-p_h)^2\rangle_{\partial\mathcal{T}_h^-}\\
&+\langle \tau_{qu}^-\tau_{qp}^- -1, (\widehat{p}_h-p_h)(\widehat{u}_h-u_h)n \rangle_{\partial\mathcal{T}_h^-}.
\end{alignat*}
It is easy to check that if the stabilization function satisfies the condition \eqref{eq:Ntau_cond}, then we get $\Phi^+\ge 0$ and $\Phi^-\ge 0$. This shows that
$$\frac{1}{2}\frac{d}{dt}\|u_h\|^2\le (f,u_h)
.$$ 

 \end{proof}

\subsection{Error analysis}
In this section, we prove the optimal error estimate for the projections of the errors in Theorem \ref{thm:new} for linear equations with $F=0$. First, we obtain the equations for the projection of the errors.

\subsubsection{The error equations}
From the equations
defining the HDG method, \eqref{eq:Nmethod1}--\eqref{eq:Nmethod3},
and the fact that the exact solution also satisfy these equations, we obtain the following error equations
\begin{alignat*}{1}
(e_q,{v})  + (e_u,v_x)  - \langle \widehat{e}_u,{v} n\rangle & =0,\\
(e_p,{z}) + (e_q,z_x)  -\langle \widehat{e}_q,{z}n \rangle & =0,\\
({e_u}_t,{w}) -(e_p,w_x) + \langle \widehat{e}_p,{w} n\rangle  &=0,
\end{alignat*}
for all $(v, z, w)\,\in\,W_h^{k}\times W_h^{k} \times W_h^{k}$, where $\widehat{e}_\omega=\omega-\widehat{\omega}_h$ for $\omega=u, q$, and $p$.
From \eqref{eq:Nmethod4} and \eqref{eq:Nmethod5}, it is easy to see that
\begin{equation*}
\begin{cases}
\widehat{e}_p^{\,+}=e_p^+ +\tau_{pu}^+\,(\widehat{e}_u - e_u^+)\,n^+,&\\
\widehat{e}_q^{\,+}=e_q^+  + \tau_{qu}^+\,(\widehat{e}_{u} - e_u^+)\,n^+,\\
\widehat{e}_q^{\,-}=e_q^-  + \tau_{qu}^-\,(\widehat{e}_u - e_u^-)\,n^- +\tau_{qp}^-\,({\widehat{e}_p^{\,-}} -e_p^-)\,n^-,
\end{cases}
\end{equation*}
and
\begin{alignat*}{2}
&\langle\widehat{e}_q, \mu\,n\rangle = 0,
\quad
\langle\widehat{e}_p, \chi\,n\rangle = 0
\end{alignat*}
for all $(\mu,\chi)\in \tilde{M}_h\times M_h(0)$.
 Now we set  $$\widehat{\epsilon}_u=\widehat{e}_u\quad\textrm{ and }\quad \widehat{\epsilon}_p^{\,-}=\widehat{e}_p^{\,-},$$  and let
 \begin{equation}
 \label{eq:error_traces}
\begin{cases}
\widehat{\epsilon}_p^{\,+}= \epsilon_p^+ + \tau_{pu}^+\,(\widehat{\epsilon}_u - \epsilon_u^+)\,n^+,&\\
\widehat{\epsilon}_q^{\,+}= \epsilon_q^+ + \tau_{qu}^+\,(\widehat{\epsilon}_u - \epsilon_h^+)\,n^+,\\
\widehat{\epsilon}_q^{\,-}= \epsilon_q^- + \tau_{qu}^-\,(\widehat{\epsilon}_u - \epsilon_u^-)\,n^- +\tau_{qp}^-\,({\widehat{\epsilon}_p^{\,-}} -\epsilon_p^-)\,n^-.
\end{cases}
\end{equation}
Using the equations \eqref{eq:projpu}--\eqref{eq:projqu-}, after some simple algebra manipulations we get that
$$\widehat{\epsilon}_p^{\,+}=\widehat{e}_p^{\,+} \quad\textrm{ and } \quad\widehat{\epsilon}_q^{\,\pm}=\widehat{e}_q^{\,\pm}.$$
Therefore, by the definition of the projection $\varPi$, \eqref{eq:projq}--\eqref{eq:proju},
 we easily obtain the following equations for the projections of errors
\begin{subequations}
\label{eq:error_eqn}
\begin{alignat}{1}
\label{eq:error_eqn1}
(\epsilon_q,{v})+(\delta_q, v) + (\epsilon_u,v_x) - \langle \widehat{\epsilon}_u,{v} n\rangle& =0,\\
\label{eq:error_eqn2}
(\epsilon_p,{z})+(\delta_p, z) + (\epsilon_q,z_x) - \langle \widehat{\epsilon}_q,{z}n \rangle& =0,\\
\label{eq:error_eqn3}
(\epsilon_{ut},{w})+(\delta_{ut}, w)-(\epsilon_p,w_x) + \langle \widehat{\epsilon}_p,{w} n\rangle &=0,\\
\label{eq:error_eqn4}
\langle\widehat{\epsilon}_q, \mu \,n \rangle = 0,
\quad
\langle\widehat{\epsilon}_p, \chi\,n\rangle &= 0
\end{alignat}
for all $(v, z, w,\mu,\chi)\,\in\,W_h^{k }\times W_h^{k } \times W_h^{k }\times \tilde{M}_h\times M_h(0)$.
\end{subequations}

\subsubsection{Energy identities}

To prove the $L^2$-error estimate in Theorem \ref{thm:new},
we begin by establishing a key identity involving the quantity
\[
\|\epsilon\|^2:=\|\epsilon_u\|^2+\|\epsilon_q\|^2+\|\epsilon_p\|^2+\|\epsilon_{ut}\|^2
\]
by energy arguments.

\begin{lemma}
\label{lemma:energy1}
We have that
$$ \frac{1}{2}\frac{d}{dt}\|\epsilon\|^2+ S+\Psi=0,$$
where
\begin{alignat*}{1}
 S=&(\delta_{ut},\epsilon_u)-(\delta_q, \epsilon_p)+(\delta_p, \epsilon_q)
    +({\delta_q}_t, \epsilon_q)+(\delta_p, \epsilon_{ut})-(\delta_{ut},\epsilon_p)\\
   &+({\delta_p}_t,\epsilon_p) -({\delta_u}_t, {\epsilon_q}_t) +({\delta_q}_t,\epsilon_{ut})
    +(\delta_{utt},\epsilon_{ut}) -({\delta_q}_t, {\epsilon_p}_t)+({\delta_p}_t, {\epsilon_q}_t),\\
\Psi=& -\langle \widehat{\epsilon}_p-\epsilon_p, (\widehat{\epsilon}_u-\epsilon_u)n\rangle
    +\frac{1}{2}\langle (\widehat{\epsilon}_q-\epsilon_q)^2, n\rangle\\
  & +\langle\widehat{\epsilon}_q-\epsilon_q, (\widehat{\epsilon}_{ut}-\epsilon_{ut})n\rangle
    +\frac{1}{2}\langle (\widehat{\epsilon}_p-\epsilon_p)^2, n\rangle\\
  & +\langle \widehat{\epsilon}_{qt}-\epsilon_{qt}, (\widehat{\epsilon}_p-\epsilon_p)n\rangle
    +\frac{1}{2}\langle (\widehat{\epsilon}_{ut}-\epsilon_{ut})^2, n\rangle\\
  & -\langle\widehat{\epsilon}_{pt}-\epsilon_{pt},(\widehat{\epsilon}_{ut}-\epsilon_{ut})n\rangle
    +\frac{1}{2}\langle (\widehat{\epsilon}_{qt}-\epsilon_{qt})^2,n\rangle\\
  &  +\frac{1}{2}\widehat{\epsilon}_q^{\;2}(x_0)+\frac{1}{2}(\widehat{\epsilon}_p+\widehat{\epsilon}_{qt})^2(x_0)-\frac{1}{2}\widehat{\epsilon}_p^{\;2}(x_N).
\end{alignat*}

\end{lemma}

\begin{proof}
Differentiating the error equations \eqref{eq:error_eqn1}--\eqref{eq:error_eqn3} with respect to $t$, we get
\begin{subequations}
\label{eq:error_eqn_dt}
\begin{alignat}{1}
\label{eq:error_eqn_dt1}
({\epsilon_q}_t,{v})+({\delta_q}_t, v) + ({\epsilon_u}_t,v_x) - \langle \widehat{\epsilon}_{ut},{v} n\rangle& =0,\\
\label{eq:error_eqn_dt2}
({\epsilon_p}_t,{z})+({\delta_p}_t, z) + ({\epsilon_q}_t,z_x) -\langle \widehat{\epsilon}_{qt},{z}n \rangle & =0,\\
\label{eq:error_eqn_dt3}
({\epsilon_u}_{tt},{w})+({\delta_u}_{tt}, w)-({\epsilon_p}_t,w_x)+ \langle \widehat{\epsilon}_{pt},{w} n\rangle &=0.
\end{alignat}
\end{subequations}
Next, we use \eqref{eq:error_eqn} and \eqref{eq:error_eqn_dt} to get four energy identities.

(i) Taking $w=\epsilon_u,  v=-\epsilon_p,$ and $z=\epsilon_q$ in \eqref{eq:error_eqn} and
adding the three equations together, we have
\begin{alignat*}{1}
0=&({\eps_u}_t,\epsilon_u)+(\delta_{ut},\epsilon_u)-(\epsilon_p, {\eps_u}_x) + \langle \widehat{\eps}_p, \eps_u n\rangle\\
&-(\eps_q, \eps_p)- (\delta_q, \eps_p)-(\eps_u, {\eps_p}_x) + \langle \widehat{\eps}_u, \eps_p n\rangle\\
&+(\eps_p, \eps_q)+(\delta_p, \eps_q)+ (\eps_q, {\eps_q}_x)-\langle \widehat{\eps}_q, \eps_q n\rangle.
\end{alignat*}
Using integration by parts, \eqref{eq:error_eqn4}, and the fact that
$$\widehat{\eps}_u|_{\partial\Omega}=\widehat{e}_u|_{\partial\Omega}=0, \qquad\widehat{\eps}_q|_{\partial\Omega_N}=\widehat{e}_q|_{\partial\Omega_N}=0,$$
 we get
\begin{alignat}{1}\label{eq:error_energy1}
\begin{split}
0=&\frac{1}{2}\frac{d}{dt}\|\eps_u\|^2 +({\delta_u}_t,\epsilon_u) -(\delta_q,\eps_p)+(\delta_p,\eps_q)\\
 &-\langle \widehat{\eps}_p-\eps_p, (\widehat{\eps}_u-\eps_u)n \rangle +\frac{1}{2}\langle (\widehat{\eps}_q-\eps_q)^2, n\rangle+\frac{1}{2}\widehat{\epsilon}_q^{\;2}(x_0).
\end{split}
\end{alignat}

(ii)  Similar to (i), taking $v=\eps_q$ in \eqref{eq:error_eqn_dt1}, $z={\eps_u}_t$ in \eqref{eq:error_eqn2}, and $w=-\eps_p$ in \eqref{eq:error_eqn3} and adding the three equations together, we get
\begin{alignat}{1}
\label{eq:error_energy2}
\begin{split}
0=& \frac{1}{2}\frac{d}{dt} \|\eps_q\|^2 + ({\delta_q}_t,\eps_q) +(\delta_p, {\eps_u}_t) -({\delta_u}_t, \eps_p)\\
 & +\langle \widehat{\eps}_q -\eps_q, (\widehat{\eps}_{ut}-\eps_{ut})n\rangle
   +\frac{1}{2}\langle (\widehat{\eps}_p -\eps_p)^2, n\rangle-\frac{1}{2}  \widehat{\epsilon}_p^{\;2}(x_N)+\frac{1}{2}  \widehat{\epsilon}_p^{\;2}(x_0).
\end{split}
\end{alignat}

(iii) 
Taking $v={\eps_u}_t$ in \eqref{eq:error_eqn_dt1},  $z=\eps_p$ in \eqref{eq:error_eqn_dt2}, and $w=-{\eps_{q}}_t$ in \eqref{eq:error_eqn3} and adding the equations together, we get
\begin{alignat}{1}
\label{eq:error_energy3}
\begin{split}
0=& \frac{1}{2}\frac{d}{dt} \|\eps_p\|^2 + ({\delta_p}_t,\eps_p) +({\delta_{q}}_t, {\eps_u}_t) -({\delta_u}_t, {\eps_{q}}_t)\\
 & +\langle \widehat{\eps}_{qt} -\eps_{qt}, (\widehat{\eps}_{p}-\eps_{p})n\rangle
   +\frac{1}{2}\langle (\widehat{\eps}_{ut} -\eps_{ut})^2, n\rangle+\widehat{\epsilon}_{qt}\widehat{\epsilon}_p (x_0).
\end{split}
\end{alignat}

(iv) Taking $v=-{\eps_p}_t, z= {\eps_q}_t$, and $w={\eps_u}_t$ in \eqref{eq:error_eqn_dt1}--\eqref{eq:error_eqn_dt3} and adding the equations together, we get
\begin{alignat}{1}
\label{eq:error_energy4}
\begin{split}
0=& \frac{1}{2}\frac{d}{dt} \|{\eps_u}_t\|^2 + ({\delta_u}_{tt},{\eps_u}_t) -({\delta_{q}}_t, {\eps_p}_t) +({\delta_p}_t, {\eps_{q}}_t)\\
 & -\langle \widehat{\eps}_{pt} -\eps_{pt}, (\widehat{\eps}_{ut}-\eps_{ut})n\rangle
   +\frac{1}{2}\langle (\widehat{\eps}_{qt} -\eps_{qt})^2, n\rangle+\frac{1}{2}\widehat{\epsilon}_{qt}^{\;2}(x_0).
\end{split}
\end{alignat}

The proof is completed by adding the four equations \eqref{eq:error_energy1}--\eqref{eq:error_energy4} together.
\end{proof}

\subsubsection{Proof of the $L^2$-error estimate}

Using Lemma \ref{lemma:energy1}, we first get the following result.
\begin{lemma}\label{lemma:inequality}
If the stabilization function  satisfies the condition \eqref{eq:tau_cond_new}, then we have
\begin{alignat*}{1}
\|\eps(t)\|^2 \le &  \|\eps(0)\|^2 +\Theta(0)+ \int_0^t\widehat{\epsilon}_p^{\;2}(x_N)\,dt+2\,|\int_0^t S\, dt| \quad \textrm{ for } 0\le t\le T,,
\end{alignat*}
 where
 $$\Theta=\langle  \tau_{qu}^+ - \tau_{pu}^+ \tau_{qu}^+,    (\widehat{\eps}_u-\eps_u)^2  \rangle_{\partial\mathcal{T}_h^+}
          +\langle 1, \tau_{qu}^-(\widehat{\eps}_u-\eps_u)^2+\tau_{qp}^-(\widehat{\eps}_p-\eps_p)^2\rangle_{\partial\mathcal{T}_h^-},
 $$
 and $S$ is the same as in Lemma \ref{lemma:energy1}.
\end{lemma}

\begin{proof}
Using the definition of $\widehat{\eps}_p^+$ and $\widehat{\eps}_q$ in \eqref{eq:error_traces}, for the $\Psi$ term  in Lemma \ref{lemma:energy1}, we have
$$\Psi= \Psi^+ +\Psi^-,$$
where
\begin{alignat*}{1}
\Psi^+=& - \langle \tau_{pu}^+, (\widehat{\eps}_u-\eps_u)^2\rangle_{\partial\mathcal{T}_h^+}
       - \frac{1}{2}\langle (\tau_{qu}^+)^2, (\widehat{\eps}_u-\eps_u)^2\rangle_{\partial\mathcal{T}_h^+}\\
    & +\langle \tau_{qu}^+, (\widehat{\eps}_u-\eps_u)(\widehat{\eps}_u-\eps_u)_t \rangle_{\partial\mathcal{T}_h^+}
       - \frac{1}{2}\langle (\tau_{pu}^+)^2, (\widehat{\eps}_u-\eps_u)^2\rangle_{\partial\mathcal{T}_h^+}\\
    & -\langle  \tau_{pu}^+ \tau_{qu}^+, (\widehat{\eps}_u-\eps_u)(\widehat{\eps}_u-\eps_u)_t \rangle_{\partial\mathcal{T}_h^+}
      -\frac{1}{2}\langle 1, (\widehat{\eps}_{ut}-\eps_{ut})^2 \rangle_{\partial\mathcal{T}_h^+}\\
    &-\langle \tau_{pu}^+, (\widehat{\eps}_{ut}-\eps_{ut})^2 \rangle_{\partial\mathcal{T}_h^+}
          -\frac{1}{2}\langle(\tau_{qu}^+)^2, (\widehat{\eps}_{ut}-\eps_{ut})^2\rangle_{\partial\mathcal{T}_h^+}\\
    & +\frac{1}{2}\widehat{\epsilon}_q^{\;2}(x_0)+\frac{1}{2}(\widehat{\epsilon}_p+\widehat{\epsilon}_{qt})^2(x_0)\\
\intertext{and}
\Psi^-=& -\langle \widehat{\eps}_p-\eps_p, \widehat{\eps}_u-\eps_u\rangle_{\partial\mathcal{T}_h^-}
       + \frac{1}{2}\langle 1,(\tau_{qu}^-(\widehat{\eps}_u-\eps_u)+\tau_{qp}^-(\widehat{\eps}_p-\eps_p))^2\rangle_{\partial\mathcal{T}_h^-}\\
    & +\langle \tau_{qu}^-(\widehat{\eps}_u-\eps_u)+\tau_{qp}^-(\widehat{\eps}_p-\eps_p), (\widehat{\eps}_u-\eps_u)_t \rangle_{\partial\mathcal{T}_h^-}
       + \frac{1}{2}\langle 1,(\widehat{\eps}_p-\eps_p)^2\rangle_{\partial\mathcal{T}_h^-}\\
    & +\langle  \tau_{qu}^-(\widehat{\eps}_u-\eps_u)_t + \tau_{qp}^-(\widehat{\eps}_p-\eps_p)_t , \widehat{\eps}_p-\eps_p\rangle_{\partial\mathcal{T}_h^-}
      +\frac{1}{2}\langle 1, (\widehat{\eps}_{ut}-\eps_{ut})^2\rangle_{\partial\mathcal{T}_h^-}\\
    &-\langle  \widehat{\eps}_{pt}-\eps_{pt}, \widehat{\eps}_{ut}-\eps_{ut}\rangle_{\partial\mathcal{T}_h^-}
      +\frac{1}{2}\langle1,(\tau_{qu}^-(\widehat{\eps}_{ut}-\eps_{ut})+\tau_{qp}^-( \widehat{\eps}_{pt}-\eps_{pt}))^2\rangle_{\partial\mathcal{T}_h^-}\\
    &-\frac{1}{2}\widehat{\epsilon}_p^{\;2}(x_N).
      \end{alignat*}
We can rewrite the term $\Psi^+ $ as
$$\Psi^+= \Gamma_1 + \frac{1}{2}\frac{d}{dt} \Theta_1,$$
 where
 \begin{alignat*}{1}
\Gamma_1 =&   \langle -\tau_{pu}^+ -\frac{1}{2} (\tau_{qu}^+)^2 -\frac{1}{2}(\tau_{pu}^+)^2, (\widehat{\eps}_u-\eps_u)^2  \rangle_{\partial\mathcal{T}_h^+} \\
              & + \langle -\frac{1}{2}-\tau_{pu}^+ -\frac{1}{2}(\tau_{qu}^+)^2, (\widehat{\eps}_{ut}-\eps_{ut})^2 \rangle_{\partial\mathcal{T}_h^+}
              +\frac{1}{2}\widehat{\epsilon}_q^{\;2}(x_0)+\frac{1}{2}(\widehat{\epsilon}_p+\widehat{\epsilon}_{qt})^2(x_0),\\
\Theta_1=  & \langle  \tau_{qu}^+ - \tau_{pu}^+ \tau_{qu}^+,    (\widehat{\eps}_u-\eps_u)^2  \rangle_{\partial\mathcal{T}_h^+} .
\end{alignat*}
Similarly, if we assume that $\tau_{qu}^-\tau_{qp}^-=1$, after some calculations we get
$$\Psi^-=  \Gamma_2+ \frac{1}{2}\frac{d}{dt} \Theta_2-\frac{1}{2}\widehat{\epsilon}_p^{\;2}(x_N),$$
where
\begin{alignat*}{1}
\Gamma_2=& \langle (\frac{1}{2}\tau_{qu}^-)^2, (\widehat{\eps}_u-\eps_u)^2\rangle_{\partial\mathcal{T}_h^-}
             +\langle \frac{1}{2}, \Big(\tau_{qp}^-(\widehat{\eps}_p-\eps_p)+(\widehat{\eps}_u-\eps_u)_t)\Big)^2\rangle_{\partial\mathcal{T}_h^-}  \\
           & +  \langle(\frac{1}{2}\tau_{qp}^-)^2, (\widehat{\eps}_{pt}-\eps_{pt})^2\rangle_{\partial\mathcal{T}_h^-}
             +  \langle \frac{1}{2}, \Big( (\widehat{\eps}_p-\eps_p)+  \tau_{qu}^-(\widehat{\eps}_u-\eps_u)_t \Big)^2\rangle_{\partial\mathcal{T}_h^-}
             ,\\
\Theta_2= &\langle 1, \tau_{qu}^-(\widehat{\eps}_u-\eps_u)^2+\tau_{qp}^-(\widehat{\eps}_p-\eps_p)^2\rangle_{\partial\mathcal{T}_h^-}.
\end{alignat*}
So from Lemma \ref{lemma:energy1} we get
\begin{equation}
\label{eq:error_sum}
\frac{1}{2}\frac{d}{dt}(\|\eps\|^2 +\Theta_1 +\Theta_2) +\Gamma_1+\Gamma_2
= \frac{1}{2}\widehat{\epsilon}_p^{\;2}(x_N)-S.
\end{equation}
Now we integrate the equation \eqref{eq:error_sum} with respect to $t$ and get
\begin{alignat*}{1}
&\frac{1}{2}\Big(\|\eps(t)\|^2+ \Theta_1(t)+\Theta_2(t)\Big)+\int_0^t (\Gamma_1+\Gamma_2)dt\\
 = & \frac{1}{2}\Big(\|\eps(0)\|^2 +\Theta_1(0)+\Theta_2(0)\Big)+\frac{1}{2}\int_0^t\widehat{\epsilon}_p^{\;2}(x_N)dt-\int_0^t S \,dt.
\end{alignat*}
It is easy to check that if $\tau_{qu}^\pm, \tau_{pu}^+$ and $\tau_{qp}^-$ satisfy the condition \eqref{eq:tau_cond_new},
we have $$\Theta_1\ge 0, \;\;\Theta_2\ge 0,\;\; \Gamma_1\ge 0,\;\; \Gamma_2\ge 0 \;\; \textrm{ for any } t\in [0, T].$$
Therefore,
\begin{alignat*}{1}
\|\eps(t)\|^2 \le &  \|\eps(0)\|^2 +\Theta(0)+ \int_0^t\widehat{\epsilon}_p^{\;2}(x_N)\,dt+2\,|\int_0^t S \,dt|,
\end{alignat*}
where $\Theta=\Theta_1+\Theta_2.$
\end{proof}

To prove Theorem \ref{thm:new}, we also need the following Lemma for error estimates of the initial approximations at $t=0$ (See Theorem 2.2 and Theorem 2.3 in \cite{ChenCockburnDongHDG16}).
\begin{lemma}\label{lemma:inital_approximation1}
If $\tau_{qu}^\pm, \tau_{pu}^+, \tau_{qp}^-$ satisfy the condition \eqref{eq:cond_tau}, then for $k>0$,
\begin{alignat*}{1}
&\|\eps_u(0)\|+\|\eps_q(0)\|+\|\eps_p(0)\|\le C h^{k+2},\\ 
&\|\widehat{e}_u(0)\|_{\Eh}+\|\widehat{e}_q(0)\|_{\Eh}+\|\widehat{e}_p(0)\|_{\Eh}\le Ch^{2k+1}.
\end{alignat*}
\end{lemma}

In addition, let us get an estimate for $\eps_{ut}$ at $t=0$.
\begin{lemma}\label{lemma:inital_approximation2}
If $\tau_{qu}^\pm, \tau_{pu}^+, \tau_{qp}^-$ satisfy  the condition \eqref{eq:cond_tau}, then for $k>0$
\begin{alignat*}{1}
&\|\eps_{ut}(0)\| \le C h^{k+1}.
\end{alignat*}
\end{lemma}

\begin{proof}
Taking $t=0$ and $w={\eps_u}_t(0)$ in the error equation \eqref{eq:error_eqn3}, we have
\begin{equation*}
(\epsilon_{ut}(0),{\eps_u}_t(0))+(\delta_{ut}(0), {\eps_u}_t(0))-(\epsilon_p(0),{{\eps_u}_t}_x(0)) + \langle \widehat{\epsilon}_p(0),{\eps_u}_t(0) n\rangle =0.
\end{equation*}
By Cauchy inequality, trace inequality and inverse inequality,
we get
$$\|{\eps_u}_t(0)\|^2\le C\|\delta_{ut}(0)\|^2+Ch^{-2}\|\eps_p(0)\|^2 +Ch^{-1}\|\widehat{\eps}_p(0)\|_{\Eh}^2.
$$
Then the conclusion  follows by  using Lemma \ref{lemma:projapprox} and Lemma \ref{lemma:inital_approximation1}.

\end{proof}

Now let us finish the proof of Theorem \ref{thm:new} by estimating the right hand side of the inequality in Lemma \ref{lemma:inequality} and using Lemma \ref{lemma:inital_approximation1} and Lemma \ref{lemma:inital_approximation2}.
\begin{proof}
We first estimate the term $\int_0^t\;\widehat{\epsilon}_p^{\;2}(x_N) dt$. Taking $\omega$ to be  $\omega_1:=\frac{x-x_0}{x_N-x_0}$ in \eqref{eq:error_eqn3},
 we get
$$\widehat{\epsilon}_p(x_N)=-(\epsilon_{ut}, \omega_1)-(\delta_{ut},\omega_1)+(\epsilon_p, \frac{1}{x_N-x_0})$$
by the fact that $\omega_1(x_0)=0$ and $\omega_1(x_N)=1$.
Using Cauchy inequality, we have
\begin{alignat*}{1}
|\widehat{\epsilon}_p(x_N)|
\le &\, |(\epsilon_{ut}, \omega_1)|+|(\delta_{ut},\omega_1)|+|(\epsilon_p, \frac{1}{x_N-x_0})| \\
\le &\, C (\|\epsilon_{ut}\|  + \|\delta_{ut}\| +\|\epsilon_p\| ).
\end{alignat*}
Then by the approximation property of the projection $\Pi$ in Lemma \ref{lemma:projapprox}, we obtain
\begin{alignat}{1}\label{eq:estimate_pN}
\int_0^t\widehat{\epsilon}_p^{\;2}(x_N)\,dt \le\; & Ch^{2k+2}+\int_0^t \|\epsilon \|^2 dt.
\end{alignat}
Next, we estimate the term $|\int_0^t  S \, dt|$. Let $$S=S_1+S_2,$$
where
 \begin{alignat*}{1}
S_1=&(\delta_{ut},\epsilon_u)-(\delta_q, \epsilon_p)+(\delta_p, \epsilon_q)
    +({\delta_q}_t, \epsilon_q)+(\delta_p, \epsilon_{ut})-(\delta_{ut},\epsilon_p)\\
   & +({\delta_p}_t,\epsilon_p) +({\delta_q}_t,\epsilon_{ut})
    +(\delta_{utt},\epsilon_{ut}),\\
S_2=&-({\delta_u}_t, {\epsilon_q}_t)-({\delta_q}_t, {\epsilon_p}_t)+({\delta_p}_t, {\epsilon_q}_t).
 \end{alignat*}
 Using Cauchy inequality and the approximation property of the projection $\Pi$ in Lemma \eqref{lemma:projapprox}, we get
\begin{alignat*}{1}
\int_0^t |S_1|dt \le Ch^{k+1}\int_0^t\|\eps \|dt.
\end{alignat*}
Integrating $S_2$ with respect to $t$, we have
\begin{alignat*}{1}
\int_0^t S_2 dt= & - ({\delta_u}_{t}, \eps_q)|_0^t + \int_0^t({\delta_u}_{tt},\eps_q)dt -({\delta_q}_t, {\epsilon_p})|_0^t+\int_0^t({\delta_q}_{tt}, {\epsilon_p}) dt\\
                & +({\delta_p}_t, {\epsilon_q})|_0^t -\int_0^t({\delta_p}_{tt}, {\epsilon_q})dt.
\end{alignat*}
By the approximation property of the projection $\Pi$ in Lemma \ref{lemma:projapprox},
\begin{alignat*}{1}
\biggl|\int_0^t S_2 dt\bigg| \le & Ch^{2k+2}+C\|\eps(0)\|^2+\frac{1}{4}\|\eps(t)\|^2+Ch^{k+1}\int_0^t \|\eps \| dt.
\end{alignat*}
So  we get
\begin{alignat}{1}\label{eq:estimate_S}
\begin{split}
\Big|\int_0^t S dt\Big|&\le \int_0^t|S_1|dt +|\int_0^t S_2\, dt|\\
&\le  Ch^{2k+2}+C\|\eps(0)\|^2+\frac{1}{4}\|\eps(t)\|^2+Ch^{k+1}\int_0^t \|\eps \| dt.
\end{split}
\end{alignat}
Applying \eqref{eq:estimate_pN} and \eqref{eq:estimate_S} to Lemma \ref{lemma:inequality}, we have
\begin{alignat*}{1}
\|\eps(t)\|^2
 \le & C \|\eps(0)\|^2 +C \Theta(0)+C h^{2k+2}+C\int_0^t\|\eps(t)\|^2 dt.
\end{alignat*}
Since
\[
\Theta(0)\le C (\|\widehat{\eps}_u(0)\|_{\Eh}^2+\|\widehat{\eps}_p(0)\|_{\Eh}^2) +Ch^{-1}(\|\eps_u(0)\|^2+\|\eps_p(0)\|^2)
\]
by Lemma \ref{lemma:inital_approximation1} and the trace inequality,  we have
\begin{alignat*}{1}
\|\eps(t)\|^2 \le & Ch^{2k+2}+C\int_0^t\|\eps(t)\|^2 dt
\end{alignat*}
using Lemma \ref{lemma:inital_approximation1} and Lemma \ref{lemma:inital_approximation2}.
Now we use Gr\"{o}nwall's inequality and get
$$\|\eps(t)\|^2\le C h^{2k+2},$$
where $C$ depends on $t$ but not on $h$. This completes the proof of Theorem \ref{thm:new}.
\end{proof}

\section{Numerical Results}
\label{sec:numericaltest}
In this section, we carry out several numerical experiments to study the accuracy and capability of our HDG method. In the first and the second numerical
experiments, we examine the orders of convergence of the method for linear and nonlinear third-order problems.
In the third and the fourth experiments, we apply the method to solve some well-known dispersive wave problems. For all the experiments, we use the following second-order midpoint rule \cite{BonaChenKarakashianXing13,ChenCockburnDongDG16} for time discretization.
Let $0=t_0<t_1<\cdots<t_J=T$ be a partition of the interval $[0,T]$ and $\Delta t_j=t_{j+1}-t_j$.
For $j=0, \cdots, J-1$ and $\omega \in \{u_h, q_h, p_h\}$,  let $\omega^{j+1}\in W_h^k$ be  defined as $$\omega^{j+1}=2\omega^{j,1}-\omega^j,$$ where $\omega^{j,1}$ is the solution of the equation
$$ \frac{\omega^{j,1}-\omega^j}{\frac{1}{2}\,\Delta t_j} +(\omega^{j,1})_{xxx}+F(\omega^{j,1})_x=0.$$
The components of the stabilization function, $(\tau_{qu}^+, \tau_{pu}^+, \tau_{qu}^-, \tau_{qp}^-)$ are taken to be $(0, -1, 1, 1)$ in all the following numerical tests.

{\it Numerical experiment 1:} In this test, we use the HDG method to solve the time-dependent third-order linear problem
$$ u_t + u_{xxx} =f,$$ where $f$ is chosen so that the exact solution is $u(x,t)=\sin(x+t)$ on the domain $(x,t)\in [0, 1]\times [0, 0.1]$. The initial condition is $u_0=\sin(x)$ and the boundary conditions are $u(0,t)=\sin(t), u(1,t)=\sin(1+t)$ and $u_x(1,t)=\cos(1+t)$. 
We take {$h={2^{-n}}$ for $n=1,\cdots, 5$}. The step size for time discretization is $\Delta t=0.1*h^2$ for $k=0, 1$, and $\Delta t=0.1*h^3$ for $k=2, 3$ so that the temporal errors are very small. We compute the orders of convergence of $u_h, q_h, p_h$ at the final time $T=0.1$, and the orders we observe in the numerical experiments  are listed in Table \ref{tab:linear_order}.

Our numerical results indicate that the orders of convergence of
$(e_u,e_q,e_p)$ are optimal as predicted by the error estimate in Theorem \ref{thm:actual_error} for any $k>0$. For $k=0$, although our error analysis is inclusive,  we observe that the method converges optimally in the numerical experiment.

\begin{table}
\small
\begin{center}
\renewcommand{\arraystretch}{1.3}
\begin{tabular}{ |c|c | c | c | c | c | c | }
\hline
$k$ & ${e_u}$ & Order &${e_q}$ & Order & {$e_p$} & Order \\ \hline
\multirow{6}{*}{0}
&1.27e-01 & -    & 1.07e-01 & -    & 1.94e-01 & - \\
&6.87e-02 & 0.89 & 6.26e-02 & 0.77 & 1.13e-01 & 0.78 \\
&3.83e-02 & 0.84 & 3.52e-02 & 0.83 & 6.10e-02 & 0.89 \\
&2.08e-02 & 0.88 & 1.94e-02 & 0.86 & 3.31e-02 & 0.88 \\
&1.07e-02 & 0.96 & 1.03e-02 & 0.92 & 1.85e-02 & 0.84 \\
 \hline
\multirow{6}{*}{1}
&1.13e-02 & -    & 1.22e-02 & -    & 6.83e-03 & - \\
&3.28e-03 & 1.79 & 3.08e-03 & 1.99 & 1.90e-03 & 1.85 \\
&8.62e-04 & 1.93 & 7.69e-04 & 2.00 & 4.87e-04 & 1.97 \\
&2.17e-04 & 1.99 & 1.92e-04 & 2.00 & 1.22e-04 & 1.99 \\
&5.44e-05 & 2.00 & 4.80e-05 & 2.00 & 3.06e-05 & 2.00 \\
\hline
\multirow{6}{*}{2}
&3.66e-04 & -    & 3.27e-04 & -    & 7.41e-04 & - \\
&4.59e-05 & 2.99 & 4.33e-05 & 2.92 & 6.99e-05 & 3.41 \\
&5.71e-06 & 3.01 & 5.50e-06 & 2.98 & 1.12e-05 & 2.64 \\
&7.10e-07 & 3.01 & 6.94e-07 & 2.99 & 1.49e-06 & 2.91 \\
&8.86e-08 & 3.00 & 8.73e-08 & 2.99 & 1.90e-07 & 2.97 \\
\hline
\multirow{6}{*}{3}
&1.97e-05 & -    & 5.43e-05 & -    & 7.32e-04 & - \\
&1.05e-06 & 4.23 & 2.24e-06 & 4.60 & 8.53e-05 & 3.10 \\
&6.50e-08 & 4.01 & 7.77e-08 & 4.85 & 4.19e-06 & 4.35 \\
&4.07e-09 & 4.00 & 3.88e-09 & 4.32 & 1.86e-07 & 4.49 \\
&2.55e-10 & 4.00 & 2.32e-10 & 4.06 & 5.68e-09 & 5.03 \\
 \hline
\end{tabular}\end{center}
\vskip.5truecm
\caption{The error $(e_u, e_q, e_p)$ and their convergence orders for the linear problem in the numerical experiment 1.}
\label{tab:linear_order}
\end{table}

{\it Numerical experiment 2:} Now we use the HDG method to solve the nonlinear third-order equation
$$ u_t + u_{xxx}+(3u^2)_x =f.
$$
 The function $f$, the initial condition and the boundary conditions are chosen so that the exact solution is $u(x,t)=\sin(2x+t)$ in the domain $(x,t)\in [0, \pi]\times [0, 0.1]$.
Here, we take the stabilization function $\tau_F=3$, given that $F(u)=3u^2$ and $\frac{1}{2} |F'(u)|=3|u|\le 3$ for the solution $u$. The mesh size for the HDG method is {$h={2^{-n}}$ for $n=3,\cdots, 7$}. The step size for time discretization is $\Delta t=0.1*h^2$ for $k=0, 1$ and $\Delta t=0.1*h^3$ for $k=2, 3$ so that the temporal errors are much smaller than the spatial errors.  The orders of convergence of $u_h, q_h, p_h$ at the final time $T=0.1$    are displayed in Table \ref{tab:nonlinear_order}.
Our numerical results show that the orders of convergence of
$(e_u,e_q,e_p)$ are also optimal for any $k\ge 0$ for the nonlinear problem.

\begin{table}
\small
\begin{center}
\renewcommand{\arraystretch}{1.3}
\begin{tabular}{ |c|c | c | c | c | c | c | }
\hline
$k$ & ${e_u}$ & Order &${e_q}$ & Order & {$e_p$} & Order \\ \hline
\multirow{6}{*}{0}
&6.63e-01 & -       & 1.34e-00 & -       & 2.63e-00 & - \\
&4.08e-01 & 0.70 & 8.58e-01 & 0.64 & 1.79e-00 & 0.56 \\
&2.37e-01 & 0.78 & 5.17e-01 & 0.73 & 1.16e-00 & 0.64 \\
&1.32e-01 & 0.84 & 2.94e-01 & 0.82 & 6.78e-01 & 0.76 \\
&7.11e-02 & 0.90 & 1.59e-01 & 0.89 & 3.71e-01 & 0.87 \\
 \hline
\multirow{6}{*}{1}
&5.35e-02 & -    & 9.60e-02 & -    & 2.31e-01 & - \\
&1.29e-02 & 2.05 & 2.36e-02 & 2.03 & 5.29e-02 & 2.12 \\
&3.18e-03 & 2.02 & 5.86e-03 & 2.01 & 1.28e-02 & 2.05 \\
&7.92e-04 & 2.01 & 1.47e-03 & 2.00 & 3.17e-03 & 2.01 \\
&1.98e-04 & 2.00& 3.67e-04 & 2.00 & 7.92e-04 & 2.00 \\
\hline
\multirow{6}{*}{2}
&3.31e-03 & -       & 5.81e-03 & -    & 1.25e-02 & - \\
&4.01e-04 & 3.05 & 7.32e-04 & 2.99 & 1.61e-03 & 2.96 \\
&4.97e-05 & 3.01 & 9.20e-05 & 2.99 & 1.99e-04 & 3.01 \\
&6.20e-06 & 3.00 & 1.15e-05 & 3.00 & 2.48e-05 & 3.00 \\
&7.74e-07 & 3.00 & 1.44e-06 & 3.00 & 3.10e-06 & 3.00 \\
\hline
\multirow{6}{*}{3}
&1.54e-04 & -    & 2.81e-04 & -    & 6.52e-04 & -   \\
&9.57e-06 & 4.01 & 1.77e-05 & 3.99 & 3.82e-05 & 4.09 \\
&5.97e-07 & 4.00 & 1.17e-06 & 3.99 & 2.39e-06 &4.00 \\
&3.73e-08 & 4.00 & 6.97e-08 & 4.00 & 1.49e-07 & 4.00 \\
&2.33e-09 & 4.00 & 4.36e-09 & 4.00 & 1.03e-08 & 3.86 \\
 \hline
\end{tabular}\end{center}
\vskip.5truecm
\caption{The error $(e_u, e_q, e_p)$ and their convergence orders for the nonlinear problem in the numerical experiment 2.}
\label{tab:nonlinear_order}
\end{table}

In the previous two tests, we have observed optimal convergence rates of the HDG method for both linear and nonlinear third-order problems. In the next two
tests, we apply the method to solve the KdV equation
\begin{equation}\label{eq:kdv}
u_t+ u_{xxx}+ (3u^2)_x=0.
\end{equation}
{\it Numerical experiment 3:} In this test, we consider the KdV equation \eqref{eq:kdv} in the domain $(x,t)\in [-10,0]\times [0,2]$ with the initial condition $u_0=2\sech^2(x-4)$ and the boundary conditions $u(-10, t)=2\sech^2(-10-4t+4), \,u(0,t)=2\sech^2(-4t+4), \,u_x(0, t)= -4\sech^2(-4t+4)\tanh(-4t+4)$. The exact solution to this initial-boundary value problem is the classical {\it solitary-wave solution} \cite{BonaChenKarakashianXing13,SamiiPandaMichoskiDawson16}
$$u(x,t)=2\sech^2(x-4t+4). $$

\begin{figure}[th]
\centering
   \begin{subfigure}[b]{\textwidth}
        \includegraphics[width=.48\textwidth]{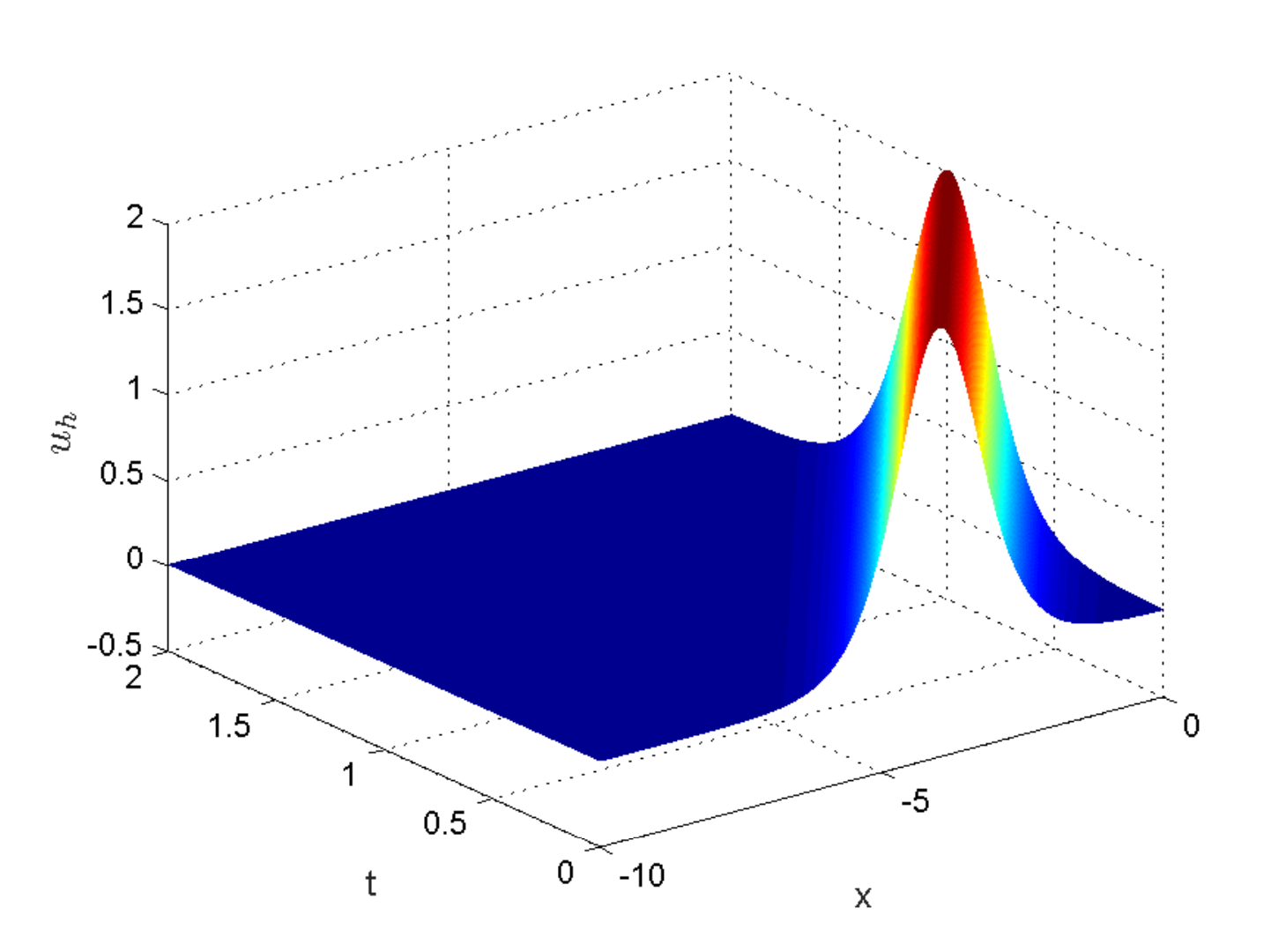}
        \includegraphics[width=.48\textwidth]{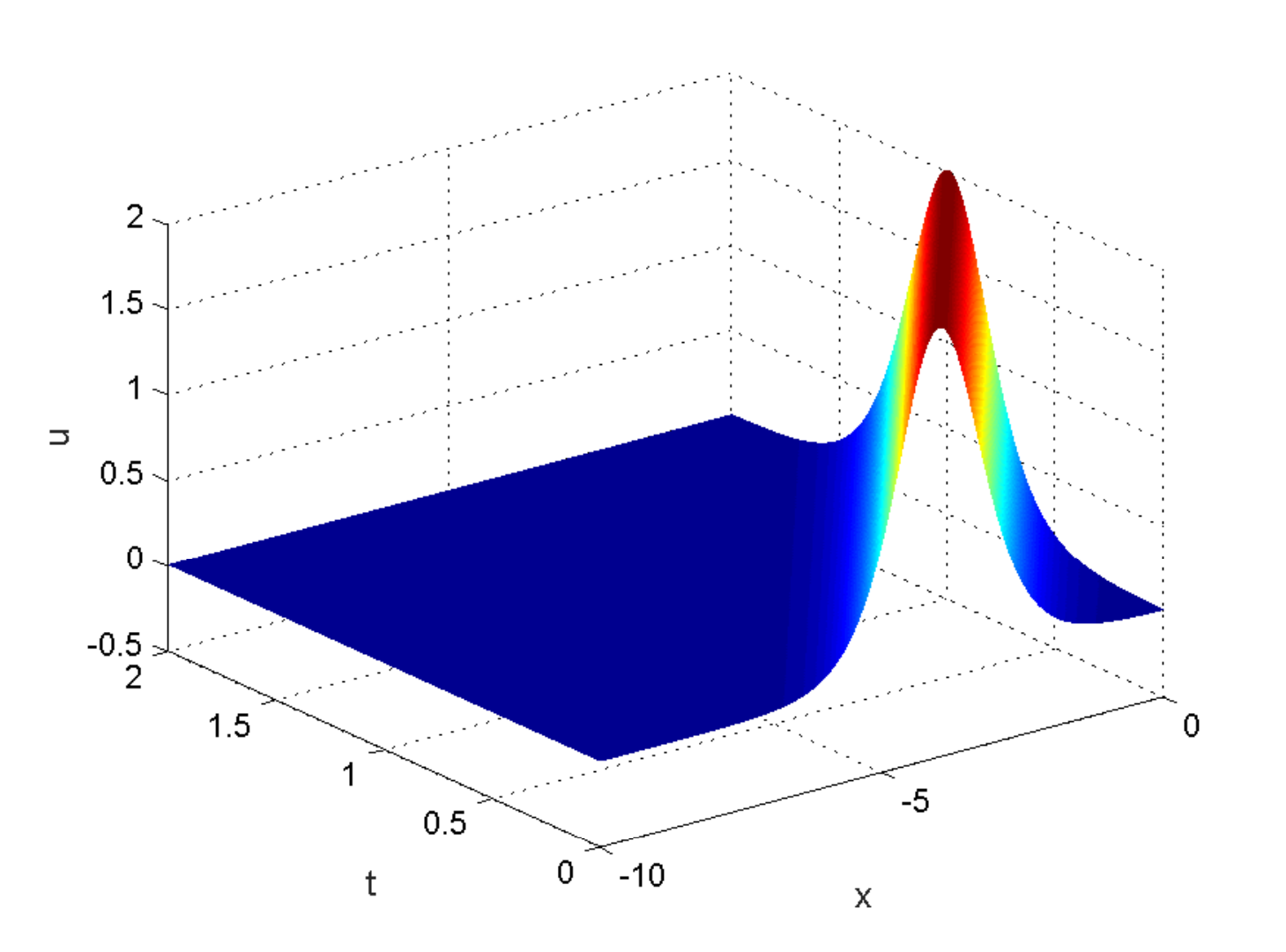}
        \caption{}
   \end{subfigure}\\
   \begin{subfigure}[b]{\textwidth}
        \includegraphics[width=.48\textwidth]{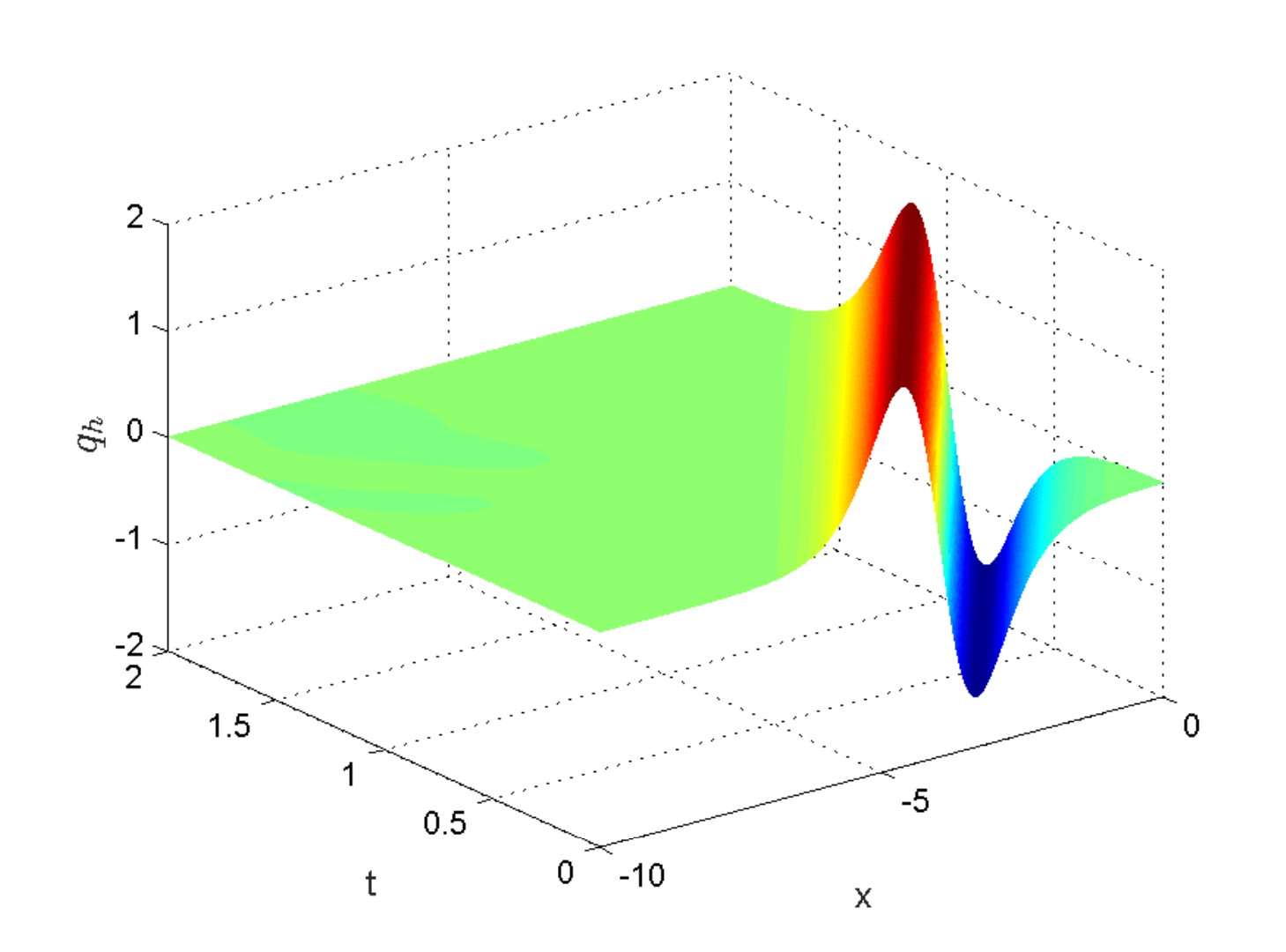}
        \includegraphics[width=.48\textwidth]{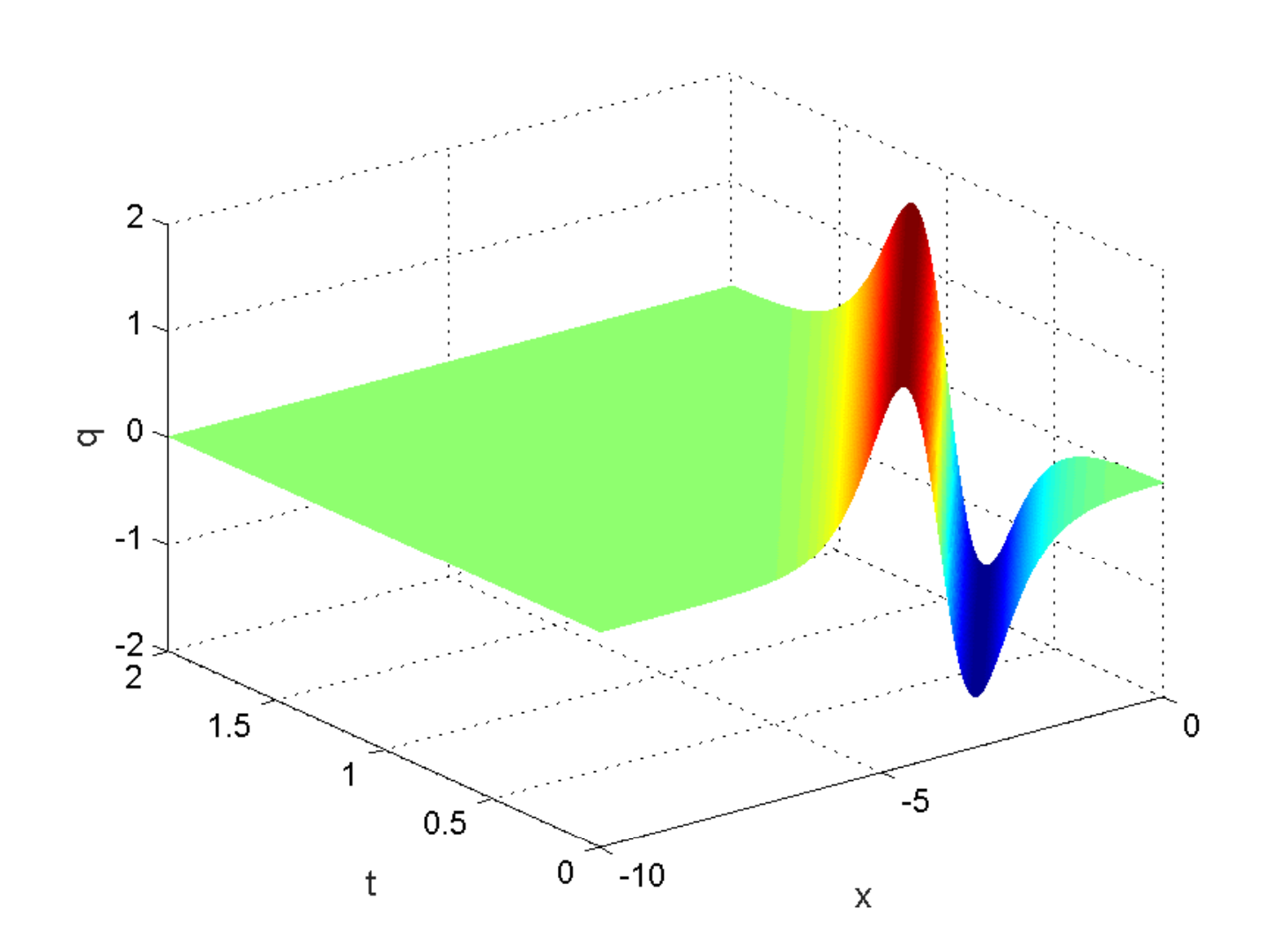}
        \caption{}
   \end{subfigure}\\
   \begin{subfigure}[b]{\textwidth}
        \includegraphics[width=.48\textwidth]{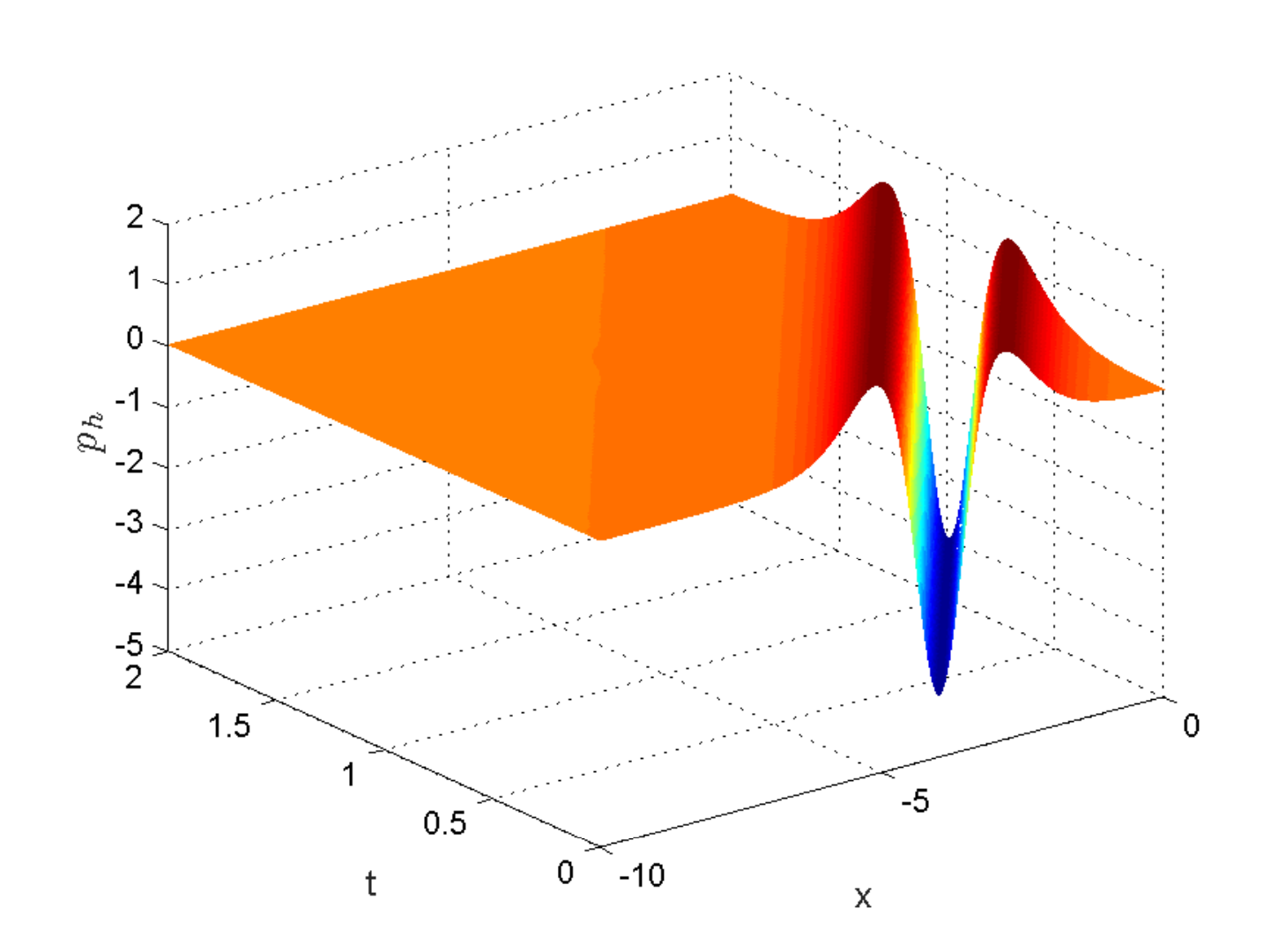}
        \includegraphics[width=.48\textwidth]{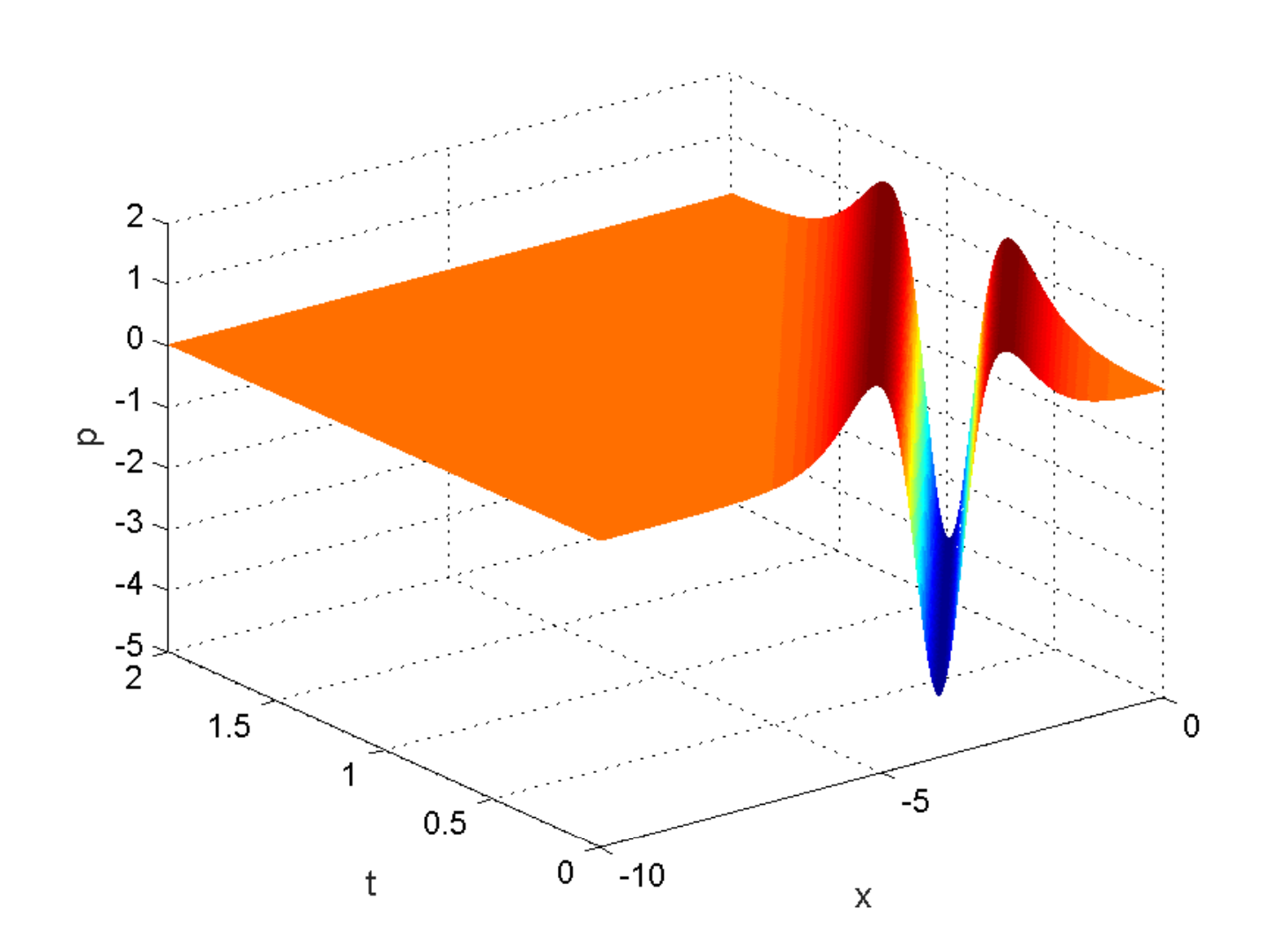}\\
        \caption{}
   \end{subfigure}
   \caption{Space-time graphs of one soliton in the domain $(x, t)\in [-10, 0]\times [0, 2]$. Evolution of the HDG approximate
solution (left) and the exact solution (right) of (A): $u$, (B): $q$, and (C): $p$.}\label{fig:1wave}
\end{figure}

In the computation, we use 100 elements, piecewise cubic polynomials, and time-step size $\Delta t=10^{-3}$, and take
$\tau_F=(F'(\widehat{u}))^2+\frac{1}{4}$ so that $\tau_F > \frac{1}{2}|F'(\widehat{u}_h)|$.
The space-time graphs of the computed solution $(u_h, q_h, p_h)$ as well as the exact solutions $(u, q, p)$ at the final time $T=2$ are displayed in Figure \ref{fig:1wave}. We observe a good match between the approximate solutions and the exact solutions.

{\it Numerical experiment 4:} In this test, we simulate the interaction of two solitary waves with different propagation speeds using our HDG method.
We consider the KdV equation \eqref{eq:kdv} in the domain $(x,t)\in [-20,0]\times [0,2]$ with the initial condition $$u_0(x)= 5 \frac{4.5 \csch^2[1.5(x+14.5)]+2\sech^2(x+12)}{\{3\coth[1.5(x+14.5)]-2\tanh(x+12)\}^2}$$
and boundary data $u(-20,t), u(0,t), u_x(0,t)$, which admits the solution (see \cite{SamiiPandaMichoskiDawson16})
$$u(x,t)=5\frac{4.5\csch^2[1.5(x-9t+14.5)]+2\sech^2(x-4t+12)}{\{3\coth[1.5(x-9t+14.5)]-2tanh(x-4t+12)\}^2}.$$

\begin{figure}[th]
\centering
   \begin{subfigure}[b]{\textwidth}
        \includegraphics[width=.48\textwidth]{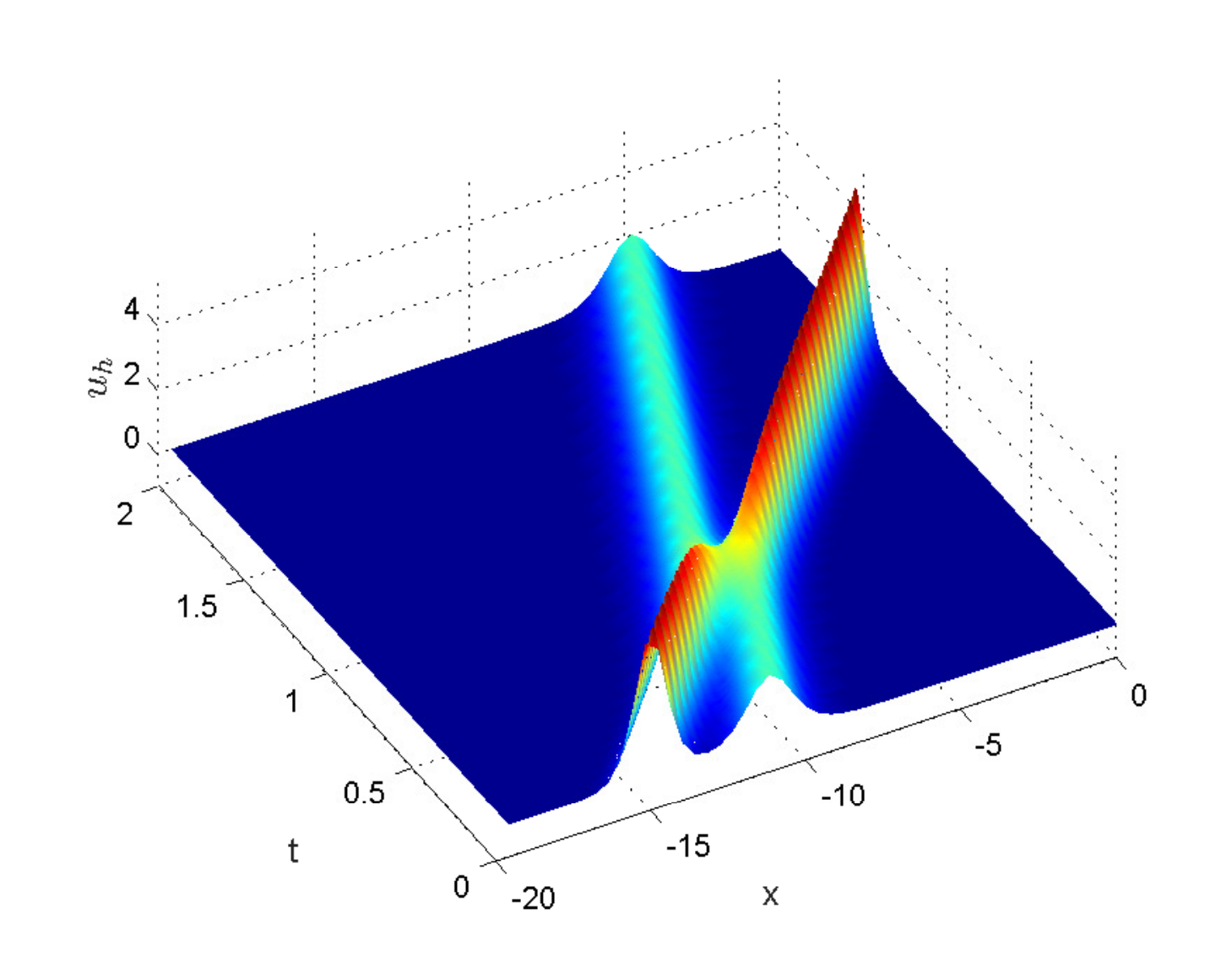}
        \includegraphics[width=.48\textwidth]{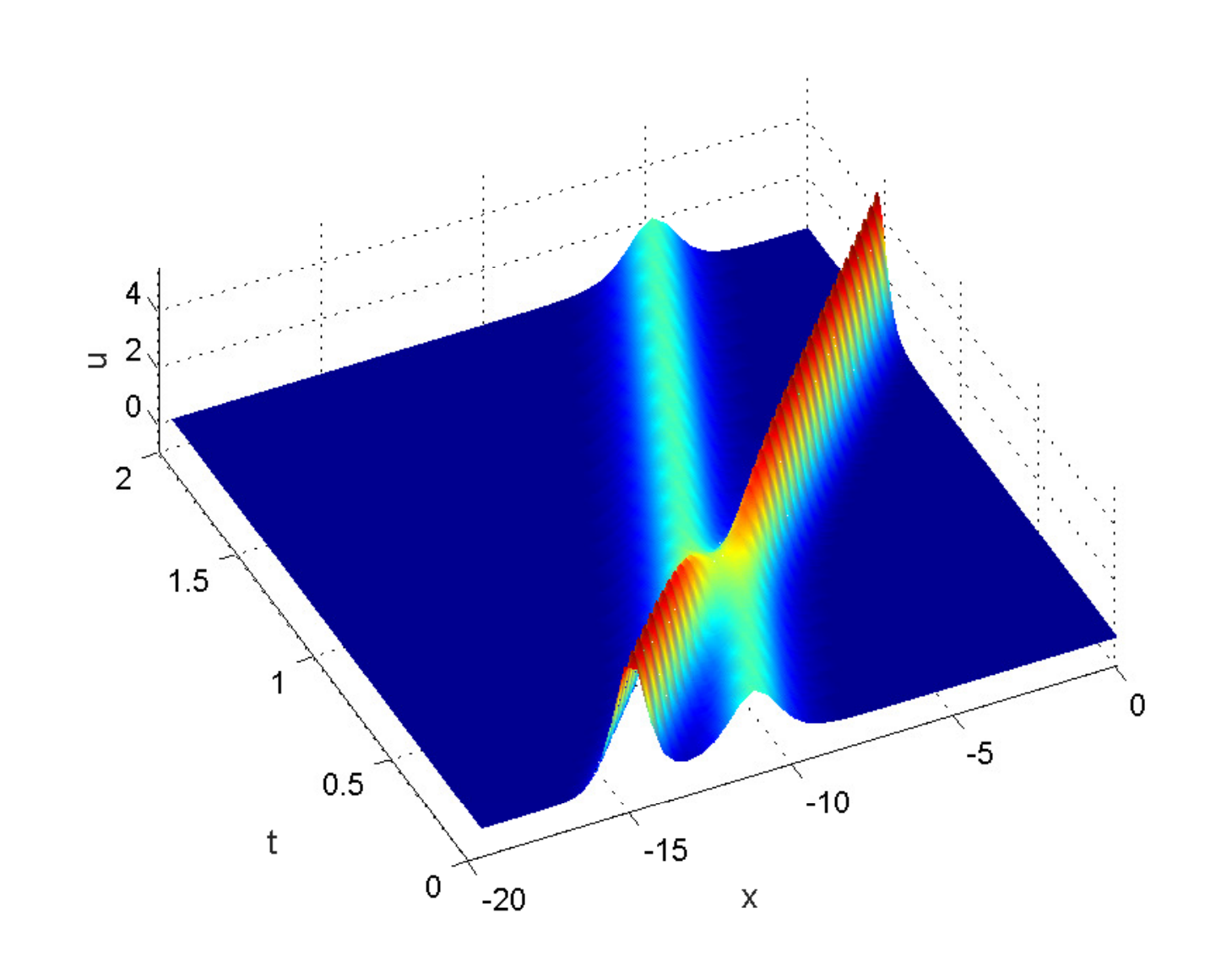}
        \caption{}
   \end{subfigure}\\
   \begin{subfigure}[b]{\textwidth}
        \includegraphics[width=.48\textwidth]{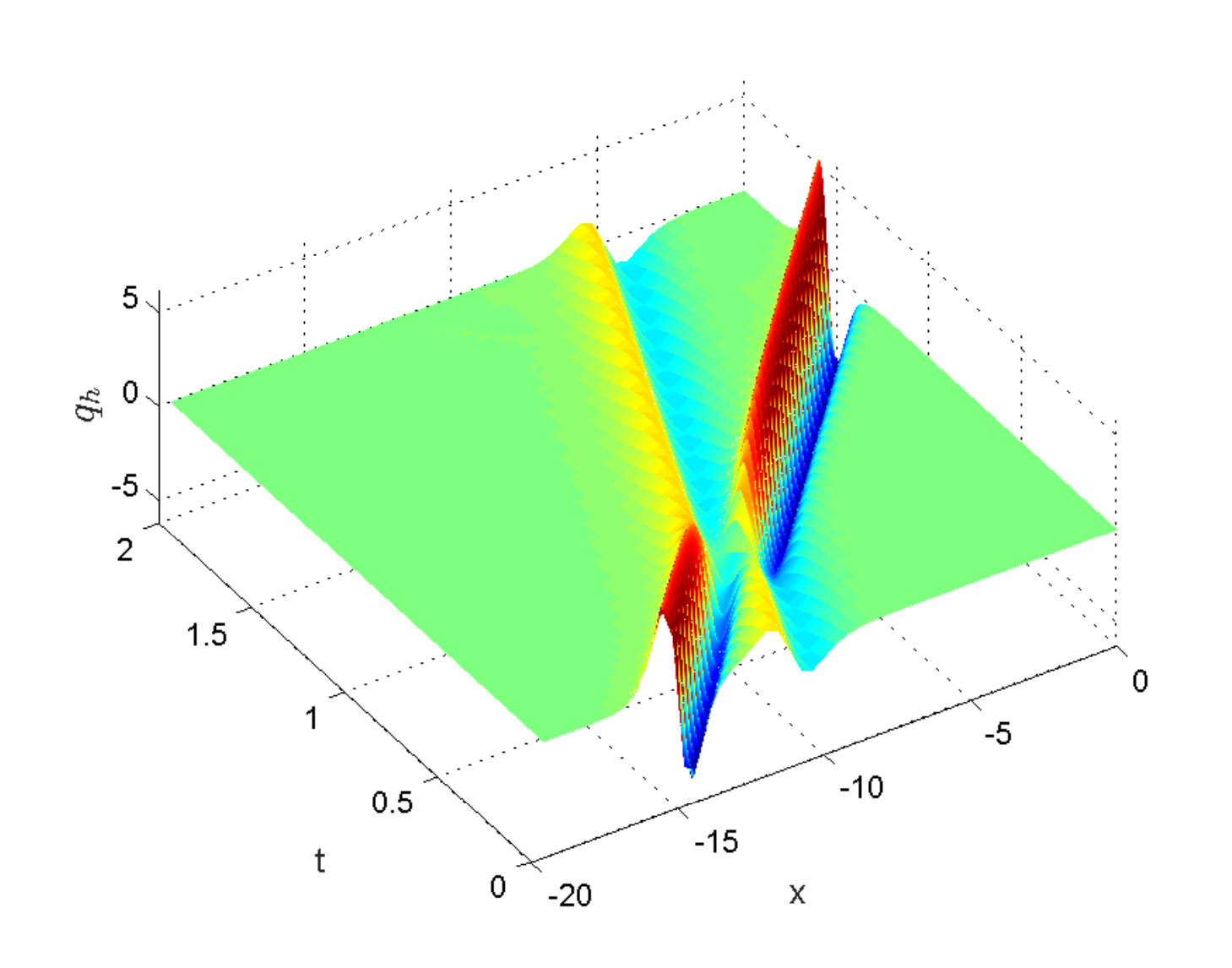}
        \includegraphics[width=.48\textwidth]{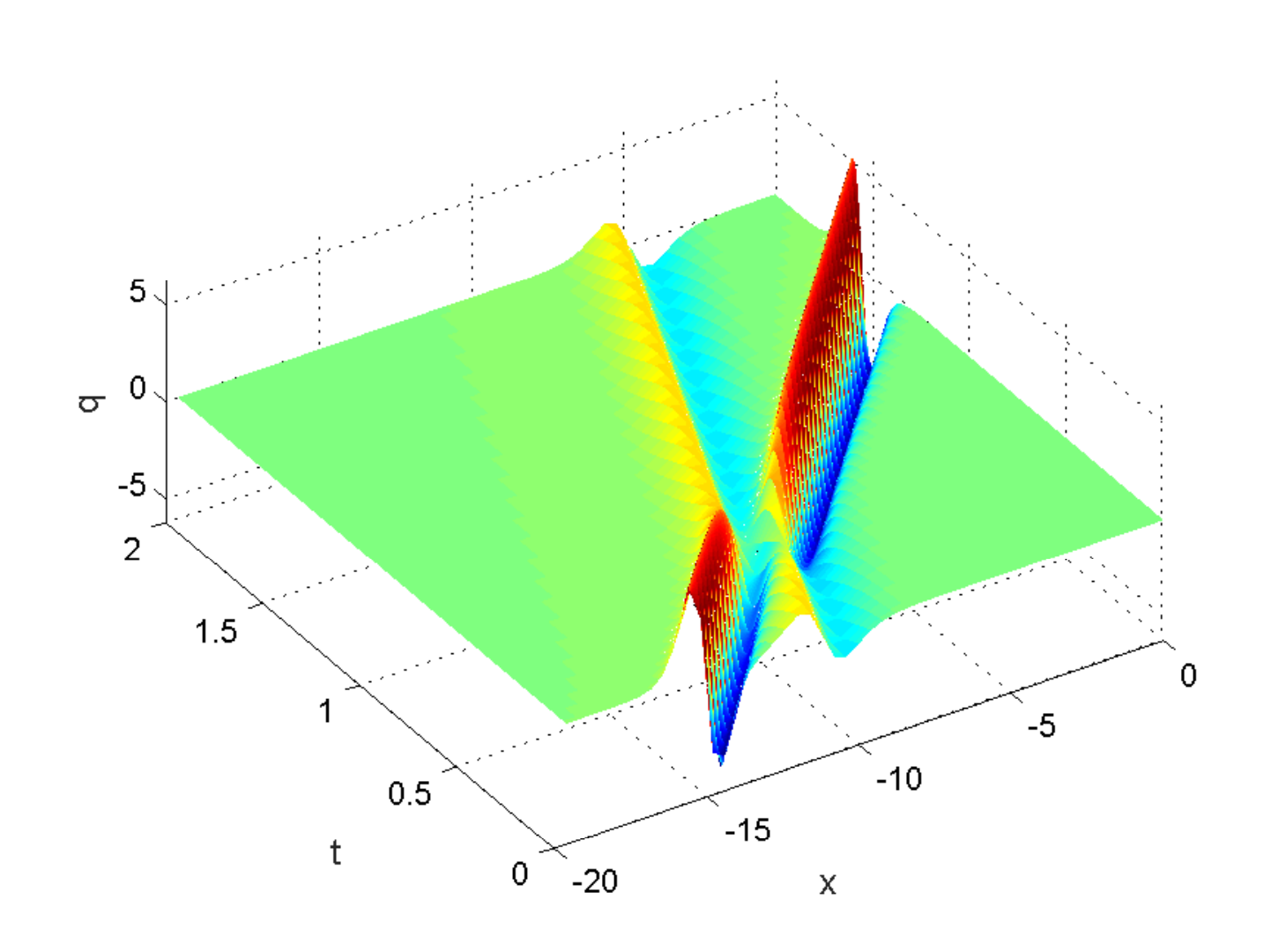}
        \caption{}
   \end{subfigure}\\
   \begin{subfigure}[b]{\textwidth}
        \includegraphics[width=.48\textwidth]{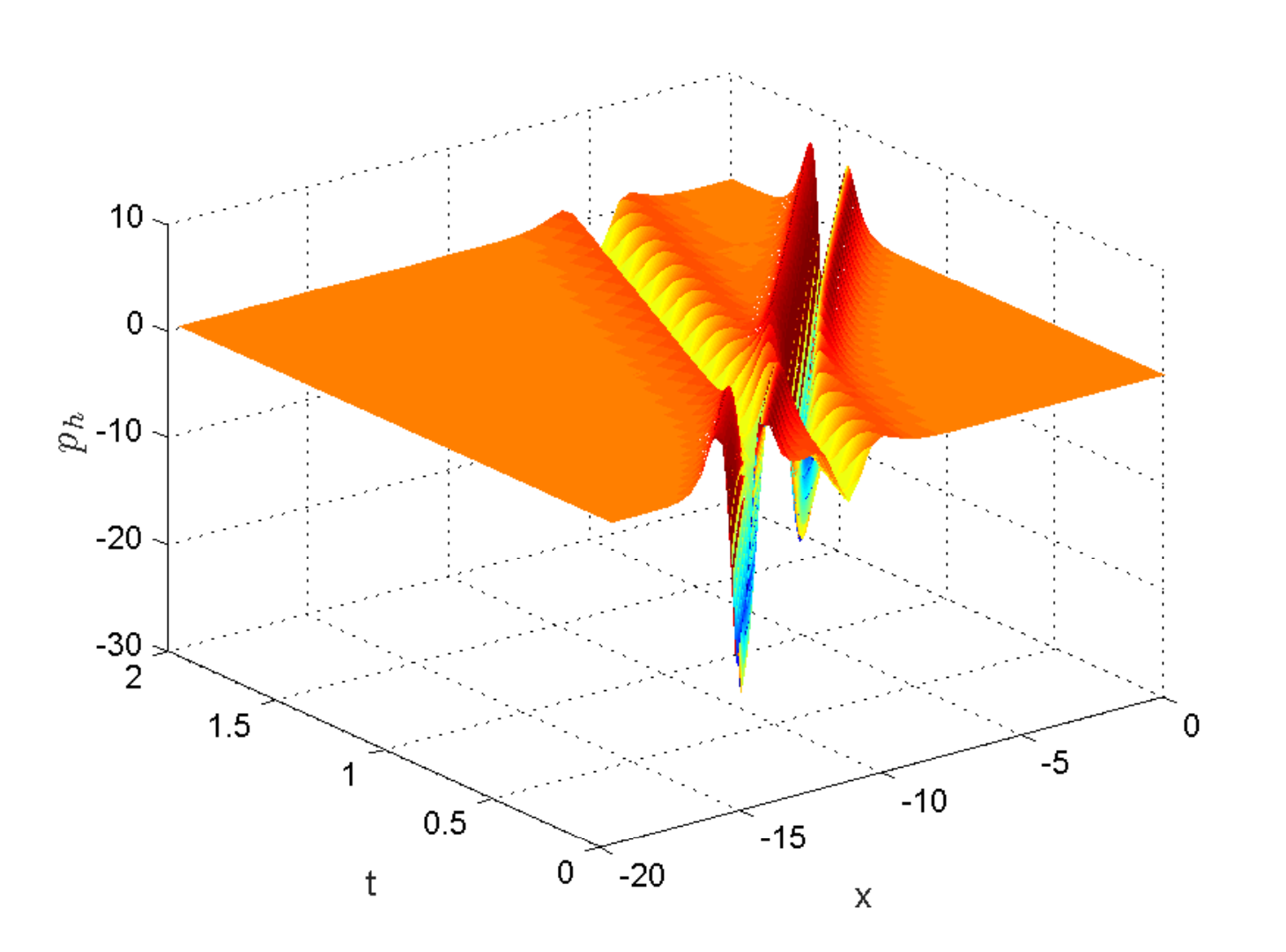}
        \includegraphics[width=.48\textwidth]{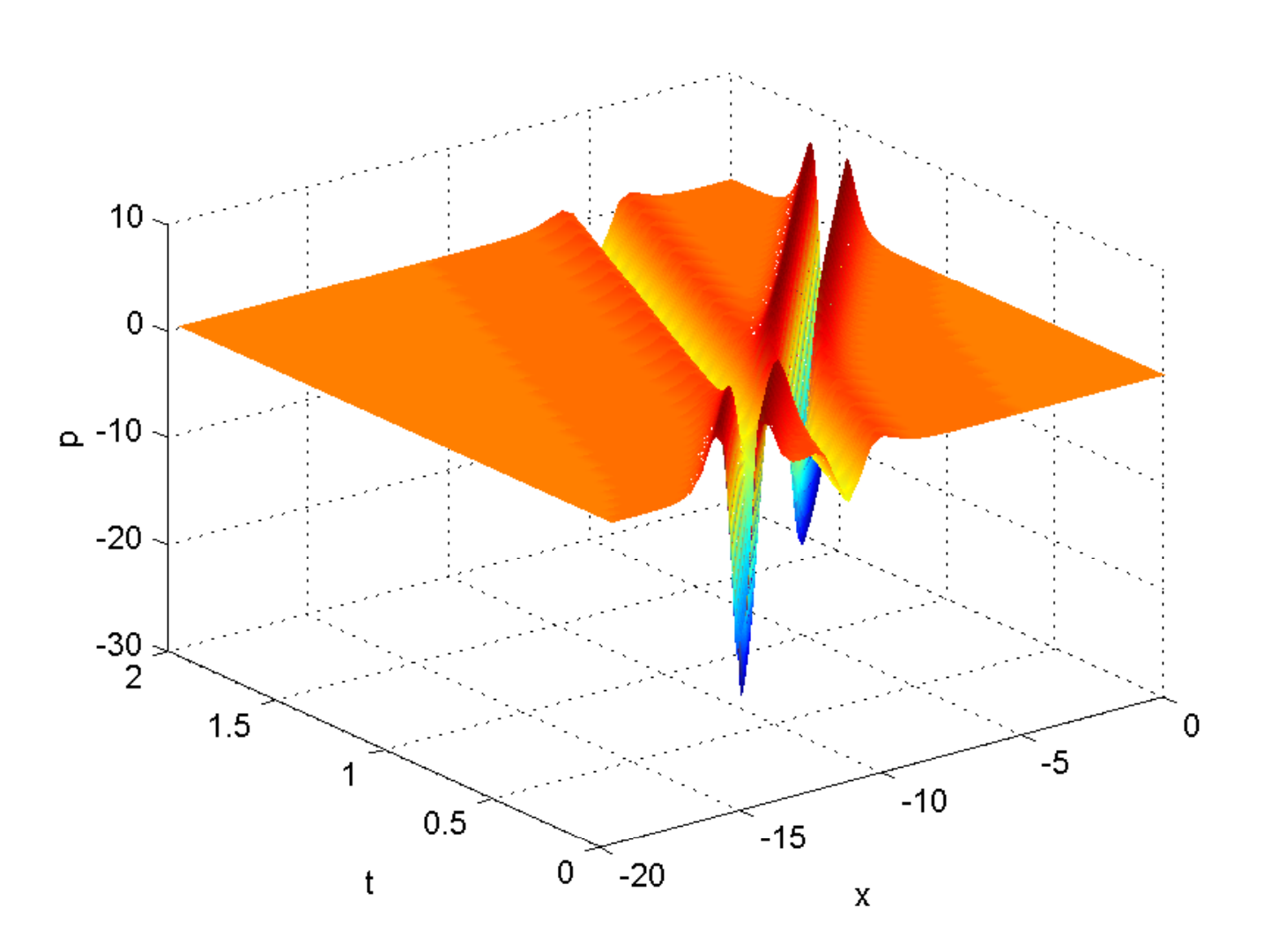}\\
        \caption{}
   \end{subfigure}
   \caption{Space-time graphs of the interaction of two solitary waves in the domain $(x, t)\in [-20, 0]\times [0, 2]$. Evolution of the HDG approximate
solution (left) and the exact solution (right) of (A): $u$, (B): $q$, and (C): $p$.}\label{fig:2wave}
\end{figure}

In our computation, we use 50 elements, piecewise cubic polynomials, and the time-step size $\Delta t=10^{-4}$.
The stabilization function $\tau_F$ is taken in the same way as in the previous test.
The space-time graphs of the HDG approximate solutions and the exact solutions are displayed in Figure \ref{fig:2wave}.
From the side-by-side comparison, we see that the HDG solutions are good approximations to the exact solutions. They show that the two waves are moving toward the same direction. The faster soliton catches up with the slower one and they overlap around $t=0.5$. Afterwards, the faster soliton
continues to propagate and the slower one falls behind.

\section{Concluding remarks}
\label{sec:conclude}
In this paper, we develop a new HDG method for time-dependent third-order equations in one space dimension based on the characterization of the exact solution as the solutions to local problems that are ``glued" together by transmission conditions.
We find conditions on the stabilization function under which the method is $L^2$ stable for KdV type equations. We also obtain optimal error estimates for the linear third-order equation. 
Numerical results from computation verify the theoretical error analysis and show that the method is able to accurately simulate solitary wave solutions of the KdV equation.
Our future work is to develop and analyze HDG methods for fifth-order KdV equations and third-order equations in multiple dimensions and complex systems.
%

\appendix
\section{Implementation}
To implement the HDG method \eqref{eq:Nmethod}, we use an implicit scheme for the discretization of the time derivative. One may use high order BDF or an implicit Runge-Kutta method for time discretization. Here, for simplicity we consider the backward Euler method with time-step $\Delta t$. At time-level $t_j$, inserting the definition of the numerical traces \eqref{eq:Nmethod4} into \eqref{eq:Nmethod1}--\eqref{eq:Nmethod4}, we obtain the equations  
\begin{alignat*}{2}
({q}_h,{v}) + (u_h,v_x)  - \langle \widehat{u}_h,{v}n \rangle   &= 0,\\
({p}_h,{z})  + (q_h,z_x)  - \langle {q}_h+ \tau_{qu} (\widehat{u}_h-u_h)n, z n \rangle  -\langle \tau_{qp} (\widehat{p}_h^{\,-}-p_h), z\rangle_{\partial {\mathcal T}_h^-} & =0,\\
\frac{1}{\Delta t}({u_h},{w}) -(p_h+F(u_h),w_x)  + \langle p_h+\tau_{pu}(\widehat{u}_h-u_h)n, w n\rangle_{\partial {\mathcal T}_h^+}\qquad\qquad&\\
+\langle \widehat{p}_h^{\,-}, {w}n \rangle_{\partial {\mathcal T}_h^-}
+\langle F(\widehat{u}_h) -\tau_{F}(\widehat{u}_h, u_h)(\widehat{u}_h -u_h)n, w n\rangle  = (f,{w}) +\frac{1}{\Delta t}(&u_h^{j-1}, w),\\
\end{alignat*}
from which $(u_h, q_h, p_h)$ can be locally solved in terms of $f$, $\widehat{u}_h$ and $\widehat{p}_h^{\,-}$,
and the equations
\begin{alignat*}{1}
\langle q_h+\tau_{qu}(\widehat{u}_h-u_h)n, \mu\, n\rangle +\langle \tau_{qp} (\widehat{p}_h^{\,-}-p_h), \mu \rangle_{\partial{\mathcal T}_h^-}&=\langle q_N, \mu\,n\rangle_{\partial\Omega_N},\\
\langle \widehat{p}_h^{\,-}, \chi\, n\rangle_{\partial{\mathcal T}_h^-}+\langle p_h+\tau_{pu}(\widehat{u}_h-u_h)n, \chi\,n\rangle_{\partial{\mathcal T}_h^+}&\\
+\langle F(\widehat{u}_h)-\tau_F(\widehat{u}_h, u_h)(\widehat{u}_h-u_h)n, \chi\,n\rangle
&=0,
\end{alignat*}
which determine the globally coupled unknowns $(\widehat{u}_h, \widehat{p}_h^{\,-})$.

Next, we apply the Newton-Raphson method to solve the above nonlinear system. 
Denoting the approximations at the current iteration by $(\bar{u}_h, \bar{q}_h, \bar{p}, \bar{\widehat{u}}_h, \bar{\widehat{p}}_h^{\,-})\in W_h^k\times W_h^k\times W_h^k\times M_h(u_D)\times \tilde{M}_h$, we want to find the increments $(\delta u_h, \delta q_h, \delta p_h, \delta \widehat{u}_h, \delta \widehat{p}_h^{\,-})  \in W_h^k\times W_h^k\times W_h^k\times M_h(0)\times \tilde{M}_h$ such that
\begin{alignat*}{1}
a_1(\delta q_h, v)+b_1(\delta u_h, v)+d_1(\delta \widehat{u}_h, v)&= r_1(v),\\
a_2(\delta p_h, z)-b_1(z, \delta q_h)+c(\delta u_h, z)+d_2(\delta\widehat{u}_h, z)+e_2(\delta\widehat{p}_h^{\,-}, z)&=r_2(z),\\
a_3(\delta u_h, w)+b_2(\delta p_h, w)+d_3(\delta\widehat{u}_h, w)+e_3(\delta\widehat{p}_h^{\,-}, w) &= r_3(w),\\
\end{alignat*}
and
\begin{alignat*}{1}
g_1(\delta p_h, \mu)+g_2(\delta q_h, \mu)+g_3(\delta u_h, \mu)+d_4(\delta\widehat{u}_h, \mu)+e_4(\delta\widehat{p}_h^{\,-}, \mu)&= r_4(\mu),\\
g_4(\delta p_h, \chi)+g_5(\delta u_h, \chi)+d_5(\delta\widehat{u}_h,\chi)+e_5(\delta\widehat{p}_h^{\,-}, \chi) &=r_5(\chi),
\end{alignat*}
for any $(v, z, w, \mu,\chi)\,\in\,W_h^{k}\times W_h^{k} \times W_h^{k}\times\tilde{M}_h\times M_h(0)$, where
\begin{alignat*}{1}
a_1(\eta, v)&=(\eta, v), \qquad
b_1(\sigma, v)=(\sigma, v_x), \qquad
d_1(\lambda, v)=-\langle \lambda, v n\rangle,\\
a_2(\rho, z)&= (\rho, z)+\langle \tau_{qp} \rho, z\rangle_{\partial\mathcal T_h^-}, \qquad
c(\sigma, z)=\langle \tau_{qu} \sigma, z \rangle,\\
d_2(\lambda, z)&= -\langle \tau_{qu}\lambda, z\rangle,\qquad
e_2(\zeta, z)=-\langle \tau_{qp} \zeta, z\rangle_{\partial\mathcal T_h^-}, \\
a_3(\sigma, w)&=\frac{1}{\Delta t}(\sigma, w)-(F'(\bar{u}_h)\sigma, w_x)-\langle \tau_{pu}\sigma, w\rangle_{\partial\mathcal T_h^+} +\langle (\bar{\tau}_F-\partial_2\bar{\tau}_F(\bar{\widehat{u}}_h-\bar{u}_h))\sigma, w\rangle,\\
b_2(\rho,w) &= -(\rho, w_x)-\langle \rho, w \rangle_{\partial\mathcal T_h^+},\qquad
e_3(\zeta, w)=\langle \zeta, w\rangle_{\partial\mathcal T_h^-},\\
d_3(\lambda, w)&=\langle (F'(\bar{\widehat{u}}_h)n-\partial_1\bar{\tau}_F(\bar{\widehat{u}}_h-\bar{u}_h)-\bar{\tau}_F)\lambda, w\rangle+\langle\tau_{pu}\lambda,w\rangle_{\partial\mathcal T_h^+},\\
g_1(\rho,\mu)&=-\langle \tau_{qp}\rho,\mu\rangle_{\partial\mathcal T_h^-},\qquad
g_2(\eta,\mu)=\langle \eta,\mu n\rangle,\qquad
g_3(\sigma,\mu)=-\langle\tau_{qu}\sigma, \mu\rangle,\\
g_4(\rho, \chi)&= \langle \rho, \chi n\rangle_{\partial\mathcal T_h^+},\qquad
g_5(\sigma, \chi)=-\langle\tau_{pu}\sigma,\chi\rangle_{\partial\mathcal T_h^+}+\langle (\bar{\tau}_F-\partial_2\bar{\tau}_F (\bar{\widehat{u}}_h-\bar{u}_h)) \sigma, \chi\rangle,\\
d_4(\lambda, \mu)&=\langle\tau_{qu}\lambda,\mu\rangle,\qquad
e_4(\zeta,\mu)=\langle \tau_{qp}\zeta, \mu\rangle_{\partial\mathcal T_h^-},\qquad
e_5(\zeta, \chi)=\langle\zeta,\chi n\rangle_{\partial\mathcal T_h^-},\\
d_5(\lambda,\chi)&=\langle \tau_{pu}\lambda, \chi\rangle_{\partial\mathcal T_h^+}+\langle (F'(\bar{\widehat{u}}_h)n-\partial_1 \bar{\tau}_F (\bar{\widehat{u}}_h-\bar{u}_h-\bar{\tau}_F))\lambda,\chi\rangle,\\
r_1(v)&=-(\bar{q}_h, v)-(\bar{u}_h, v_x)+\langle\bar{\widehat{u}}_h, vn\rangle,\\
r_2(z)&=-(\bar{p}_h,z)+(\bar{q}_{hx}, z)+\langle \tau_{qu}(\bar{\widehat{u}}_h-\bar{u}_h),z\rangle+\langle\tau_{qp}(\bar{\widehat{p}}_h^{\,-}-\bar{p}_h), z\rangle_{\partial\mathcal T_h^-},\\
r_3(w)&=(f+\frac{1}{\Delta t}(u_h^{j-1}-\bar{u}_h), w)+({\bar{p}_h}+F(\bar{u}_h), w_x)
-\langle \bar{\widehat{p}}^-_h, w\rangle_{\partial\mathcal T_h^-}\\
&\;\;\;\;-\langle \bar{p}_h n+\tau_{pu}(\bar{\widehat{u}}_h-\bar{u}_h), w  \rangle_{\partial\mathcal T_h^+}
-\langle F(\bar{\widehat{u}}_h)n-\bar{\tau}_F(\bar{\widehat{u}}_h-\bar{u}_h), w \rangle,\\
r_4(\mu) &=\langle q_N, \mu n \rangle_{\partial\Omega_N}-\langle\bar{q}_h n+\tau_{qu}(\bar{\widehat{u}}_h-\bar{u}_h), \mu\rangle-\langle\tau_{qp}(\bar{\widehat{p}}_h^{\,-} -\bar{p}_h), \mu\rangle_{\partial\mathcal T_h^-},\\
r_5(\chi) &= -\langle \bar{\widehat{p}}_h^{\,-}, \chi n\rangle_{\partial\mathcal T_h^-} -\langle \bar{p}_h n+\tau_{pu}(\bar{\widehat{u}}_h-\bar{u}_h), \chi\rangle_{\partial\mathcal T_h^+}\\
&\;\;\;\; -\langle F(\bar{\widehat{u}}_h)-\bar{\tau}_F (\bar{\widehat{u}}_h-\bar{u}_h), \chi\rangle.
\end{alignat*}
Here we have used the notation $\bar{\tau}_F:=\tau_F(\bar{\widehat{u}}_h, \bar{u}_h)$, and $\partial_1 \bar{\tau}_F$ (respectively, $\partial_2\bar{\tau}_F$) denotes the first-order partial derivative of $\tau_F$ with respect
to the first argument (respectively, second argument) evaluated at $(\bar{\widehat{u}}_h, \bar{u}_h)$.

The discretization of the system above gives rise to matrix equations of the form
\begin{equation}\label{eq:local_matrix_eq}
  \begin{bmatrix}
  0&  A_1 & B_1\\
  A_2 & -B_1^T & C \\
  B_2 &0 & A_3
  \end{bmatrix}
  \begin{bmatrix}
  \delta p_h\\
  \delta q_h\\
  \delta u_h
  \end{bmatrix}
  + \begin{bmatrix}
    D_1 & 0\\
    D_2 & E_2\\
    D_3 & E_3
    \end{bmatrix}
    \begin{bmatrix}
    \delta \widehat{u}_h\\
    \delta \widehat{p}_h^{\,-}
    \end{bmatrix}
  =\begin{bmatrix}
    R_1\\
    R_2\\
    R_3
    \end{bmatrix},
\end{equation}
and
\begin{equation}\label{eq:global_matrix_eq}
  \begin{bmatrix}
G_1 & G_2 & G_3\\
G_4 & 0   & G_5
   \end{bmatrix}
   \begin{bmatrix}
   \delta p_h\\
   \delta q_h\\
   \delta u_h
   \end{bmatrix}
   +\begin{bmatrix}
   D_4 & E_4\\
   D_5 & E_5
   \end{bmatrix}
   \begin{bmatrix}
   \delta\widehat{u}_h\\
   \delta\widehat{p}_h^{\,-}
   \end{bmatrix}
 = \begin{bmatrix}
   R_4\\
   R_5
   \end{bmatrix}.
\end{equation}
From \eqref{eq:local_matrix_eq}, we get
\begin{equation}\label{eq:local_sol}
  \begin{bmatrix}
  \delta p_h\\
  \delta q_h\\
  \delta u_h
  \end{bmatrix}
  =   \begin{bmatrix}
  0&  A_1 & B_1\\
  A_2 & -B_1^T & C \\
  B_2 &0 & A_3
  \end{bmatrix}^{-1}
  \Bigg(\begin{bmatrix}
    R_1\\
    R_2\\
    R_3
    \end{bmatrix}
   -\begin{bmatrix}
    D_1 & 0\\
    D_2 & E_2\\
    D_3 & E_3
    \end{bmatrix}
    \begin{bmatrix}
    \delta \widehat{u}_h\\
    \delta \widehat{p}_h^{\,-}
    \end{bmatrix}
    \Bigg)
    \end{equation}
 We emphasize that the above inverse can be computed on each element independently of each other since the matrices $A_1, A_2, A_3, B_1, B_2$ and $C$
 are block-diagonal owing to the discontinuous nature of the approximation spaces. Applying \eqref{eq:local_sol} to \eqref{eq:global_matrix_eq}, we get the global linear system 
 \[\mathbb{K}    \begin{bmatrix}
    \delta \widehat{u}_h\\
    \delta \widehat{p}_h^{\,-}
    \end{bmatrix}
    =\mathbb{F},
    \]
 where
 \begin{alignat*}{1}
&\mathbb{K}=\begin{bmatrix}
       D_4 & E_4\\
       D_5 & E_5
     \end{bmatrix}
   - \begin{bmatrix}
      G_1 & G_2 & G_3\\
      G_4 & 0   & G_5
    \end{bmatrix}
        \begin{bmatrix}
     0&  A_1 & B_1\\
     A_2 & -B_1^T & C \\
     B_2 &0 & A_3
    \end{bmatrix}^{-1}
    \begin{bmatrix}
    D_1 & 0\\
    D_2 & E_2\\
    D_3 & E_3
    \end{bmatrix}
\intertext{and}
  & \mathbb{F}=  \begin{bmatrix}
     R_4\\
     R_5
     \end{bmatrix}
    -\begin{bmatrix}
      G_1 & G_2 & G_3\\
      G_4 & 0   & G_5
    \end{bmatrix}
    \begin{bmatrix}
     0&  A_1 & B_1\\
     A_2 & -B_1^T & C \\
     B_2 &0 & A_3
    \end{bmatrix}^{-1}
    \begin{bmatrix}
    R_1\\
    R_2\\
    R_3
    \end{bmatrix}.
 \end{alignat*}
 Therefore, the only globally coupled degrees of freedom are those associated with $\delta\widehat{u}_h$ and $\delta\widehat{p}_h^{\,-}$, which live only on element interfaces. Due to the one-dimensional setting of the KdV equation, the size and the bandwidth of the global linear system are independent of the degrees of polynomials used; it only depends on the number of subintervals in the mesh. Once $\delta\widehat{u}_h$ and $\delta\widehat{p}_h^{\,-}$ are obtained, $(\delta p_h, \delta q_h, \delta u_h)$ can be locally computed by using \eqref{eq:local_sol}.

\bigskip
\bigskip
\noindent {\bf Acknowledgements}
 The author would like to acknowledge the support of National Science Foundation grant DMS-1419029.


\end{document}